\documentclass{article}
\usepackage{amsmath}
\usepackage{amsthm}
\usepackage{amsfonts}
\usepackage{amsbsy}
\usepackage{amssymb}
\usepackage[initials]{amsrefs}
\newtheorem{theorem}{Theorem}

\newtheorem{lemma}{Lemma}

\newcommand{\R}{{\mathbb R}}

\newcommand{\set}[2]{ \left\{ #1 \ \left| \ #2 \right. \right\} }

\usepackage{mathrsfs}

\newcommand{\ang}[1]{ \left< #1 \right>} 
\newcommand{\GL}{{\mathrm{GL}}}

\newtheorem*{theoremp}{Theorem 5}

\title{Uniform sublevel Radon-like inequalities}
\author{Philip T. Gressman}

\begin{document}
\maketitle

\begin{abstract}
This paper is concerned with establishing uniform weighted $L^p$-$L^q$ estimates for a class of operators generalizing both Radon-like operators and sublevel set operators.  Such estimates are shown to hold under general circumstances whenever a scalar inequality holds for certain associated measures (the inequality is of the sort studied by Oberlin \cite{oberlin2000II}, relating measures of parallelepipeds to powers of their Euclidean volumes).  These ideas lead to previously unknown, weighted affine-invariant estimates for Radon-like operators as well as new $L^p$-improving estimates for degenerate Radon-like operators with folding canonical relations which satisfy an additional curvature condition of Greenleaf and Seeger \cite{gs1994} for FIOs (building on the ideas of Sogge \cite{sogge1991} and Mockenhaupt, Seeger, and Sogge \cite{mss1993}); these new estimates fall outside the range of estimates which are known to hold in the generality of the FIO context.
\end{abstract}

This paper continues the study of sublevel set operators which began in \cite{gressman2010II}.  Specifically, these are operators of the form
\begin{equation} 
I_\epsilon(f_1,\ldots,f_m) := \int \chi_{|\rho(x_1,\ldots,x_m)| \leq \epsilon} \prod_{i=1}^m f_i(x_i) dx_1 \cdots dx_m. \label{mainobj}
\end{equation}
where $x_i \in U_i \subset \R^{d_i}$ and $\rho : U_1 \times \cdots \times U_m \rightarrow \R$ is real-valued and smooth.  Here the object is to establish boundedness on $L^{p_1} \times \cdots \times L^{p_m}$ along with a uniform estimate of the rate of decay of the norm as $\epsilon \rightarrow 0^+$.
The systematic study of sublevel set operators began roughly a decade ago with the works of Carbery, Christ, and Wright \cite{ccw1999}; Phong, Stein, and Sturm \cite{pss2001}; and Carbery and Wright \cite{cw2002}.  Aside from their usefulness in establishing estimates for multilinear oscillatory integrals, sublevel set operators were originally of interest because the analysis of such operators (specifically, the boundedness on products of $L^p$-spaces) appeared to be governed by the so-called Newton polytope of the phase function $\rho$.  It was demonstrated by Carbery \cite{carbery2009}, however, that the Newton polytope provides only an incomplete picture of the analysis of \eqref{mainobj}, and that there are geometrically-invariant, nonlinear expressions in $\rho$ and its derivatives which, when nonvanishing, imply additional estimates for \eqref{mainobj} which fall outside the convex hull of the estimates available in \cite{ccw1999}, \cite{pss2001}, and \cite{cw2002} (in the case of Carbery \cite{carbery2009}, the relevant quantity is the Hessian determinant of convex phases $\rho$; this and other quantities are given in \cite{gressman2010II} under the assumption that $\rho$ is polynomial or real analytic).  

The goal of this present paper is to demonstrate certain deep connections between operators of the form \eqref{mainobj}, Radon-like averaging operators, and a certain curvature condition of Oberlin which, in this context, may be interpreted as a multi-parameter scalar sublevel set inequality.  These connections lead to several new uniform weighted estimates for sublevel set operators and Radon-like averages over hypersurfaces.  Two of the most interesting examples are recorded below, phrased as isoperimetric inequalities reminiscent of the Loomis-Whitney inequality \cite{lw1949} and its generalizations \cite{bcw2005}.
\begin{theorem}
Suppose $d_l \leq d_r$.  Let $x \in \R^{d_l}$, $y \in \R^{d_r}$, and consider any $\rho$ which is a real polynomial in $x$ and $y$ and any $\varphi$ which is a coordinate-wise real polynomial mapping in $x$ and $y$ into $\R^{d_r - d_l}$.  \label{theorem1}
%Consider the set $E_{\epsilon,B}$ given by \label{theorem1}
%\[  E_{\epsilon,B} := \set{ (x,y) \in \R^{d_l} \times \R^{d_r}}{ |\rho(x,y)| \leq \epsilon \mbox{ and } \varphi(x,y) \in B}. \]
%for any $\epsilon > 0$ and any $B \subset \R^{d_r - d_l}$ which is a product of intervals.  
Let the Phong-Stein rotational curvature of the pair $\rho,\varphi$ be given by
\begin{equation} W^1_{x,y} := \det \left[ \begin{array}{ccccccc}
0 & \cdots & 0 & 0 & \frac{\partial \rho}{\partial x_1} & \cdots & \frac{\partial \rho}{\partial x_{d_l}} \\
\frac{ \partial \varphi_1}{\partial y_1} & \cdots & \frac{ \partial \varphi_{d_r-d_l}}{\partial y_1} & \frac{\partial \rho}{\partial y_1} & \frac{\partial^2 \rho}{\partial x_1 \partial y_1} & \cdots & \frac{\partial^2 \rho}{\partial x_{d_l} \partial y_1} \\
\vdots & \ddots & \vdots & \vdots & \vdots & \ddots & \vdots \\
\frac{ \partial \varphi_1}{\partial y_{d_r}} & \cdots & \frac{ \partial \varphi_{d_r-d_l}}{\partial y_{d_r}} & \frac{\partial \rho}{\partial y_{d_r}} & \frac{\partial^2 \rho}{\partial x_1 \partial y_{d_r}} & \cdots & \frac{\partial^2 \rho}{\partial x_{d_l} \partial y_{d_r}}
\end{array}
\right]. \label{radonrotcurv}
\end{equation}
%then
 Then there is a constant $C$ depending only on dimension and the degrees of $\rho$ and $\varphi$ such that
\begin{equation}
  \int_E   \left| W^1_{x,y} \right| ^{\frac{1}{d_l+1}} \left| f(x) g(y) \right| dx dy  \leq C |\rho(E)| |\varphi(E)|^{\frac{1}{d_l+1}} ||f||_{\frac{d_l+1}{d_l}} ||g||_{\frac{d_l+1}{d_l}} 
\end{equation}
for any closed set $E \subset \R^{d_l} \times \R^{d_r}$ and any measurable functions $f,g$ on $\R^{d_l}$ and $\R^{d_r}$, respectively.  Here $| \cdot |$ represents the Lebesgue measure when applied to sets.
In particular, if $\pi_L(x,y) := x$ and $\pi_R(x,y) := y$, then
\[ \int_E |W^1_{x,y}|^{\frac{1}{d_l+1}} dx dy \leq C |\rho(E)| |\varphi(E)|^{\frac{1}{d_l+1}} |\pi_L(E)|^{\frac{d_l}{d_l+1}} |\pi_R(E)|^{\frac{d_l}{d_l+1}} \]
for any closed $E \subset \R^{d_l} \times \R^{d_r}$.
\end{theorem}

\begin{theorem}
Suppose $d_l \leq d_r+1$ and let $x,y$ be as in theorem \ref{theorem1}.  Let $\rho$ be a polynomial in $x$ and $y$ and let $\varphi$ be a coordinate-wise polynomial mapping in $x$ and $y$ into $\R^{d_r - d_l + 1}$.  Let $M$ be the $d_r \times d_r$ matrix whose $(i,j)$-entry is given by \label{theorem2}
\[ M_{ij} := \det \left[ \begin{array}{ccccccc}
0 & \cdots & 0 & 0 & \frac{\partial \rho}{\partial x_1} & \cdots & \frac{\partial \rho}{\partial x_{d_l}} \\
0 & \cdots & 0 & 0 & \frac{\partial^3 \rho}{\partial x_1 \partial y_i \partial y_j } & \cdots & \frac{\partial^3 \rho}{\partial x_{d_l} \partial y_i \partial y_j} \\
\frac{ \partial \varphi_1}{\partial y_1} & \cdots & \frac{ \partial \varphi_{d_r-d_l+1}}{\partial y_1} & \frac{\partial \rho}{\partial y_1} & \frac{\partial^2 \rho}{\partial x_1 \partial y_1} & \cdots & \frac{\partial^2 \rho}{\partial x_{d_l} \partial y_1} \\
\vdots & \ddots & \vdots & \vdots & \vdots & \ddots & \vdots \\
\frac{ \partial \varphi_1}{\partial y_{d_r}} & \cdots & \frac{ \partial \varphi_{d_r-d_l+1}}{\partial y_{d_r}} & \frac{\partial \rho}{\partial y_{d_r}} & \frac{\partial^2 \rho}{\partial x_1 \partial y_{d_r}} & \cdots & \frac{\partial^2 \rho}{\partial x_{d_l} \partial y_{d_r}}
\end{array}
\right], \]
and let the weight $W^2_{x,y}$ be given as a block-form determinant by the formula
\begin{equation} W^2_{x,y} := \det \left[ \begin{array}{ccc} 0 & 0 & \partial_y \varphi \\
0 & 0 & \partial_y \rho \\
(\partial_y \varphi)^T & (\partial_y \rho)^T & M \end{array} \right] \label{radonrotcurv2}
\end{equation}
where $\partial_y \varphi$ is the $(d_r - d_l + 1) \times d_r$ matrix whose $(i,j)$-entry is $\frac{\partial \varphi_i}{\partial y_j}$ and so on. Then
\begin{equation}
  \int_{E}  \left| W^2_{x,y} \right| ^{\frac{1}{d_l(d_l-1)}} \left| f(x) g(y) \right| dx dy  \leq C |\rho(E)| |\varphi(E)|^{\frac{1}{d_l-1}} ||f||_{\frac{d_l^2 - d_l }{d_l^2 - 2 d_l + 2}} ||g||_{\frac{d_l-1}{d_l-2}} 
\end{equation}
holds uniformly for all measurable $f,g$ and closed $E$ with a constant $C$ depending only on the degrees of the polynomials and the dimensions $d_l,d_r$.
In particular, 
\[ \int_E |W^2_{x,y}|^{\frac{1}{d_l(d_l-1)}} dx dy \leq C |\rho(E)| |\varphi(E)|^{\frac{1}{d_l-1}} |\pi_L(E)|^{1- \frac{1}{d_l}\frac{d_l-2}{d_l-1}} |\pi_R(E)|^{\frac{d_l-2}{d_l-1}} \]
for any closed $E \subset \R^{d_l} \times \R^{d_r}$.
\end{theorem}

It is worth noting that theorems \ref{theorem1} and \ref{theorem2} are both affine invariant, and, as such, fall in the sharpest possible class of estimates which hold for degenerate geometric averaging operators.  Affine-invariant estimates have an extensive history in the study of the Fourier restriction problem going back to Drury \cite{drury1990}; the reader is advised to see the introduction of \cite{dlw2009} for further references.  
In the case of averaging operators, however, uniform, affine-invariant estimates have previously been available only in a few cases.  For averages over curves, see the work of Oberlin \cite{oberlin1999II}, and Dendrinos, Laghi, and Wright \cite{dlw2009}.  When the average is given by convolution with an affine surface measure, see Oberlin \cite{oberlin2000II}.  In this latter work, Oberlin demonstrates that affine surface measure is essentially the largest possible weight which guarantees the standard $L^{\frac{n+1}{n}} \rightarrow L^{n+1}$ estimate holds (meaning that any other measure of smooth density yielding this inequality will be dominated by a constant times the affine surface measure).  In this context, theorem \ref{theorem2} demonstrates that, while the $L^{\frac{n+1}{n}} \rightarrow L^{n+1}$ is necessarily tied to a specific, affine-invariant weight, there are nevertheless other estimates which are connected to entirely different affine-invariant weights.

\subsection{Applications to averages over hypersurfaces}

The weight \eqref{radonrotcurv} is called the Phong-Stein rotational curvature because it corresponds exactly to the usual definition when $d_l = d_r$.  This is no coincidence; the uniform sublevel set estimates in theorems \ref{theorem1} and \ref{theorem2} have immediate implications for Radon-like operators.  % , as will be explained now.
Let $U \subset \R^{d_l} \times \R^{d_r}$ be open and let $\rho : U \rightarrow \R$ be a polynomial.  One may consider the hypersurface ${\mathcal M} := \set{(x,y) \in U}{ \rho(x,y) = 0}$ to be the incidence relation of a Radon-like transform $R$ which averages functions on $\R^{d_r}$ over hypersurfaces; specifically, if one defines
\[ \Sigma_x^{R,c} := \set{ y \in \R^{d_r} }{ (x,y) \in U, \ \rho(x,y) = c, \ \partial_y \rho(x,y) \neq 0 }\]
%\[ \Sigma_x := \set{ y \in \R^{d_r}}{ (x,y) \in U, \ \rho(x,y) = 0 }, \]
and considers the Radon-like operator $R^c_E$ given by
\begin{equation} R^c g(x) := \int_{\Sigma_x^{R,c}} g(y) \chi_E(x,y) \frac{ d {\mathcal H}^{d_r-1}(y)}{|\partial_y \rho(x,y)|} \label{radondef}
\end{equation}
(where the measure is the $(d_r-1)$-dimensional Hausdorff measure and $\partial_y \rho$ is the Euclidean gradient of $\rho$ with respect to $y$ only), then the coarea formula (see Federer \cite{federer1969}) dictates that
\[ \int_{-\epsilon}^{\epsilon} \int f(x) R^c g(x) dx dc = \int_{|\rho(x,y)| \leq \epsilon} \chi_E(x,y) f(x) g(y) dx dy. \]
In particular, when $f$ and $g$ are suitably regular and nonnegative, it is easy to check that
\[ \int f(x) R^0 g(x) dx \leq  \liminf_{\epsilon \rightarrow 0^+} \frac{1}{2 \epsilon} \int_{|\rho(x,y)| \leq \epsilon} \chi_E(x,y) f(x) g(y) dx dy \]
(the interested reader is referred to the appendix for more details).
Thus, theorems \ref{theorem1} and \ref{theorem2} have the following consequences: %immediate analogues for Radon-like operators:
\begin{theorem}
Under the same hypotheses as theorems \ref{theorem1} and \ref{theorem2}, the operators
\begin{align*}
R^0_{W^1} g(x) & :=  \int_{\Sigma_x^{R,0}} g(y) \chi_E(x,y) |W^1_{x,y}|^{\frac{1}{d_l+1}} \frac{ d {\mathcal H}^{d_r-1}(y)}{|\partial_y \rho(x,y)|}, \\
R^0_{W^2} g(x) & :=  \int_{\Sigma_x^{R,0}} g(y) \chi_E(x,y) |W^2_{x,y}|^{\frac{1}{d_l(d_l-1)}} \frac{ d {\mathcal H}^{d_r-1}(y)}{|\partial_y \rho(x,y)|},
\end{align*}
satisfy the inequalities
\begin{align}
 || R^0_{W^1} g ||_{d_l+1} & \leq C |\varphi(E)|^{\frac{1}{d_l+1}} ||g||_{\frac{d_l+1}{d_l}} \label{rad1}, \\
  || R^0_{W^2} g ||_{\frac{d_l(d_l-1)}{d_l-2}} & \leq C |\varphi(E)|^{\frac{1}{d_l-1}} ||g||_{\frac{d_l-1}{d_l-2}}, \label{rad2}
\end{align}
where for \eqref{rad1}, the constant $C$ is exactly as described in theorem \ref{theorem1} and likewise for \eqref{rad2} and theorem \ref{theorem2}.  Recall that $\varphi$ maps into $\R^{d_r-d_l}$ in the former case and into $\R^{d_r-d_l - 1}$ in the latter.
 \end{theorem}
In this context, the weights $W^1_{x,y}$ and $W^2_{x,y}$ have familiar geometric interpretations.
When $R^c$ is realized as a Fourier integral operator of order $-\frac{d-1}{2}$ (see Guillemin and Sternberg \cite{gs1977} or Phong and Stein \cite{ps1986I} for the calculations), its canonical relation ${\mathscr C} \subset T^* \R^d \setminus 0 \times T^* \R^d \setminus 0$ is given by 
\[ {\mathscr C}^c := \set{ ((x, \lambda \partial_x \rho(x,y)),(y,-\lambda \partial_y \rho(x,y)))}{(x,y) \in U, \ \lambda \in \R \setminus 0, \ \rho(x,y) = c}; \]
in this case, the nonvanishing of the weight $W^1_{x,y}$ corresponds to the situation in which ${\mathscr C}^c$ is locally the graph of a canonical transformation. 
The idea behind \eqref{rad1}, namely, damping by the rotational curvature, is certainly not new.  Aside from the work on the Fourier restriction problem (beginning with Drury \cite{drury1990}), Phong and Stein \cite{ps1994}, \cite{ps1998}, prove damped oscillatory integral estimates which may be used to establish $L^2$-Sobolev estimates in the semi-translation-invariant case.  Oberlin \cite{oberlin2000II} also proves a version of \eqref{rad1} when $g$ is a characteristic function and $R^0_{W^1}$ is translation-invariant.  Related ideas have appeared in the work Sogge and Stein \cite{ss1990}, Christ \cite{christ1984}, and many others.
Along these same lines, there have been many works to explicitly establish $L^p$-improving estimates for Radon-like operators in various circumstances, including Littman \cite{littman1971}; Oberlin and Stein \cite{os1982}; Oberlin \cite{oberlin1987}, \cite{oberlin1997}; Phong and Stein \cite{ps1991}; Greenleaf and Seeger \cite{gs1994}, \cite{gs1998}; 
Iosevich and Sawyer \cite{is1996}
Greenleaf, Seeger, and Wainger \cite{gsw1998}, \cite{gsw1999}; Seeger \cite{seeger1998}; 
Christ \cite{christ1998}; Bak \cite{bak2000}; Bak, Oberlin, and Seeger \cite{bos2002}; Tao and Wright \cite{tw2003}; Lee \cite{lee2003}, \cite{lee2004}; and Yang \cite{yang2005}.  This listing is not exhaustive and does not include the (even larger) body of work devoted to $L^2$- and $L^p$-Sobolev estimates, which would indirectly imply (usually non-sharp) $L^p$-improving estimates as well.
Even so, it appears that the present paper is the first place in which the well-known estimate \eqref{rad1} is established uniformly for a broad class of functions $\rho$ with strong control of the constant.

%% Drury, Oberlin, Christ (k-plane transform and determinant stuff)

%% Optimal regularity in Sobolev spaces for FIOs with one-sided folds Comech \cite{comech1999}

%% averages over hypersurface regularity Sogge and Stein \cite{ss1990}

%% \cite{drury1990} First to use damping (in case of curves).

%% \cite{ps1994} \cite{ps1998} damped oscillatory integral version.  The former also implicity establishes damped $L^2$-Sobolev estimates for semi-translation invariant operators in $\R^2$.  It's not clear though if anybody's ever proved a general non-uniform $L^p$-$L^q$ result in this context...

%% Oberlin \cite{oberlin2000II} establishes a uniform restricted weak-type version of \eqref{rad1} when $R^0$ is given by convolution with a surface measure.

 Likewise, it is relatively easy to see that the nonvanishing of $W^2_{x,y}$ corresponds to the situation where the hypersurfaces 
\[ {\mathscr L}^b := \set{ ((x,\xi),(y,\eta)) \in {\mathscr C}^c}{ \varphi(x,y) = b} \]
satisfy the property that, for fixed $x$, the projection of ${\mathscr L}^b$ onto $T_x^* \R^d$ is locally a hypersurface and has $d-2$ nonvanishing principal curvatures (the maximum number possible since the hypersurface is homogeneous).  In particular, if $W^1_{x,y} = 0$ at some point and one chooses $\varphi := W^1_{x,y}$, then the nonvanishing of $W^2_{x,y}$ indicates the presence one-sided Whitney fold for which the fold hypersurface (coinciding with ${\mathscr L}^0$ in this case) satisfies a maximal curvature condition of exactly the sort studied by Greenleaf and Seeger \cite{gs1994}, \cite{gs1998}; this curvature condition is a natural outgrowth of work by Sogge \cite{sogge1991} and Mockenhaupt, Seeger, and Sogge \cite{mss1993} on the local smoothing problem (which were, in turn, inspired by the curvature conditions appearing in the work of Carleson and Sj\"{o}lin \cite{cs1972}).  In the present case, for the particular choice $\varphi := W^1_{x,y}$ when $d_l = d_r = d$, theorem 1.3 of \cite{gs1994} implies that $R^0$ locally maps $L^2$ to $L^{2 d}$ near any pair $(x,y)$ where $W^2_{x,y} \neq 0$.  By virtue of theorems \ref{theorem1} and \ref{theorem2} and Bourgain's trick, the following immediate improvement is possible:
\begin{theorem}
If the weight $W^2_{x,y}$, given by \eqref{radonrotcurv2} taking $\varphi := W^1_{x,y}$, satisfies $W^2_{x,y} \geq \beta$ for all $(x,y) \in U$ then \label{greenleafseeger}
\[ || R^0 g||_{\frac{d^2}{d-1},\infty} \leq C \beta^{-\frac{1}{2 d^2}} || g||_{\frac{d}{d-1},1} \]
with the constant $C$ again depending only on the degree of $\rho$ and the dimension (the spaces are the usual Lorentz spaces).
\end{theorem}
\begin{proof}
The proof rests on the inequality which holds for all $\epsilon > 0$:
\[ \chi_{|W^2_{x,y}| \geq \beta} \leq |\epsilon^{-1} W^1_{x,y}|^{\frac{1}{d+1}} + |\beta^{-1} W^2_{x,y} |^{\frac{1}{d(d-1)}} \chi_{|W^1_{x,y}| \leq \epsilon}. \]
In particular,
\begin{align*}
 \int \chi_{F}(x) R^0 \chi_G(x) dx \leq & \ \epsilon^{-\frac{1}{d+1}} \int \chi_{F}(x) R^0_{W^1} \chi_G(x) dx \\ & + \beta^{-\frac{1}{d(d-1)}} \int \chi_{F}(x) R^0_{W^2} \chi_G(x) dx. 
\end{align*}
Applying the inequalities \eqref{rad1} and \eqref{rad2} and optimizing the right-hand side over the choice of $\epsilon$ establishes the result.
\end{proof}

\subsection{Geometric characterizations}
Theorems \ref{theorem1} and \ref{theorem2} are themselves consequences of the following theorem, which is of a similar spirit to the work of Oberlin \cite{oberlin2000II}.  Specifically, the reader's attention is drawn to corollary 3 of \cite{oberlin2000II}, which establishes that, for any surface measure $\mu$ on $\R^n$, 
\[ || \mu \star \chi_E ||_{n+1} \leq C |E|^{\frac{n}{n+1}} \ \forall E \Leftrightarrow \mu(R) \leq C' |R|^{\frac{n-1}{n+1}} \ \forall \mbox{ rectangles } R \subset \R^n. \]
Oberlin establishes this result by considering the special case when $\mu$ is given by affine surface measure (which is obtained by multiplying the Euclidean surface measure times the absolute value of the Gaussian curvature raised to the power $\frac{1}{n+1}$) and showing that this particular measure is in some sense the ``largest'' measure which satisfies the rectangle condition.  The converse follows by testing the convolution directly on rectangles $R$.

This rectangle condition appears to be fundamentally connected to many of the most important geometric problems arising in harmonic analysis.  Aside from its appearance in the context of averaging operators \cite{oberlin2000II}, Bak, Oberlin, and Seeger \cite{bos2008} have shown that the same condition implies Fourier restriction estimates when $\mu$ is supported on a curve; more recently, Bak and Seeger \cite{bs2010} have proved general Fourier restriction estimates for measures using only assumptions on the decay of the Fourier transform coupled with a geometric measure inequality very similar to Oberlin's rectangle condition.  It has also been shown that this rectangle condition is equivalent to the boundedness of certain multilinear determinant functionals \cite{gressman2010} which are essentially ubiquitous in the literature, appearing in, for example, \cite{fefferman1970}, \cite{zygmund1974}, \cite{prestini1979}, \cite{christ1984}, \cite{christ1985}, \cite{dm1985}, \cite{dm1987}, \cite{drury1988}, \cite{valdimarsson}.  This same determinant functional also arises in the case of averaging operators over hypersurfaces, and so the characterization \cite{gressman2010} turns out to be a key tool which allows one to establish uniform $L^p$--$L^q$ estimates directly by testing an appropriate measure on rectangles:
\begin{theorem}
Let $\rho : U \rightarrow \R$, where $U$ is an open subset of $\R^{d_l} \times \R^{d_r}$. For any $x \in \R^{d_l}$, let $U_x := \set{y \in \R^{d_r}}{(x,y) \in U}$.  Fix a closed set $E \subset U$ and suppose that there exists a weight $w(x,y)$ for which the uniform geometric estimate 
\begin{equation} \sup_{x \in \R^{d_l}} \sup_{Q \in \GL(d_l,\R)} |Q|^{\frac{s}{d_l}} \int_{U_x} \frac{|w(x,y)| \chi_{E}(x,y) }{|Q \partial_x \rho(x,y)|^s} dy  \leq K_E |\rho(E)|  \label{geomassump}
\end{equation}
holds.  Then there is a constant $C$ such that 
%  Then when
%\[ E_{\epsilon,B} := \set{ (x,y) \in U}{ |\rho(x,y)| \leq \epsilon \mbox{ and } \varphi(x,y) \in B} \]
%and $\epsilon \leq \epsilon_0$, one has the uniform estimate 
\begin{equation} 
\int_E |w(x,y)|^{\frac{1}{s+1}} |f(x)g(y)| dx dy \leq C K_E^{\frac{1}{s+1}} |\rho(E)|  ||f||_{p'} ||g||_{\frac{s+1}{s}}  \label{geomineq}
\end{equation}
 holds for all $f,g$, where \label{geomthm}
\[ \frac{1}{p'} = 1 - \frac{1}{d_l} \frac{s}{s+1} \]
and the constant $C$ equals
 $C_{d_l} N_L^{\frac{s}{d_l(s+1)}}$ for some fixed $C_{d_l}$ depending only on dimension and some constant $N_L$ which counts nondegenerate solutions of a system of equations on $U$ associated to the function $\rho$.
\end{theorem}
Note that the assumption \eqref{geomassump} is exactly the statement that the measures $\mu_x$ supported on cones in $\R^d$ given by
\[ \int f(z) d \mu_x(z) := \int_0^\infty \int_{U_x} \chi_{E}(x,y) |w(x,y)| f(\lambda \partial_x \rho(x,y)) \lambda^{s-1} dy d \lambda \]
uniformly satisfy an Oberlin-type condition on rectangles which are centered at the origin.  Theorem \ref{geomthm2} provides an alternative to the condition \eqref{geomassump} which is not posed on cones but requires a slightly stronger uniformity condition.  Both theorems \ref{geomthm} and \ref{geomthm2} will be used in section \ref{finalsec} to prove $L^p$--$L^q$ estimates for averages over hypersurfaces in various degenerate situations.

Aside from theorem \ref{geomthm} and its companion, theorem \ref{geomthm2}, the second major concern of this paper is to begin a direct and in-depth study of the measure condition \eqref{geomassump}.  The condition \eqref{geomassump} measures the ``curvature'' of the measures $\mu_x$ in an affine-invariant way.  Despite its simplicity relative to \eqref{geomineq}, the inequality \eqref{geomassump} is nevertheless still surprisingly difficult to establish directly outside of certain ``trivial'' cases.  It turns out, though, that one can also establish inequalities of the form \eqref{geomassump} given information like that contained in \eqref{geomineq}.  As a result, it is possible to generate a number of new, nontrivial geometric estimates of the form \eqref{geomassump} by an inductive procedure (which is closely related to the inductive procedure from \cite{gressman2010II}).  This inductive procedure is carried out in section \ref{inductionsec}, while section \ref{specificsec} is devoted to the analysis of some of the simpler products of the induction procedure, which are computed explicitly in lemmas \ref{codim1lemma}, \ref{codim2lemma}, \ref{codim2klemma}, and \ref{codim3lemma}.

\subsection{The finite multiplicity condition}
Regarding the nature of the constants appearing in theorem \ref{geomthm} and, by extension, theorems \ref{theorem1} and \ref{theorem2}, a definition is in order.
Given a finite collection of function $f_1,\ldots,f_k$ defined on a domain $U \subset \R^d$, it will be necessary to consider the algebraic properties of mappings generated by these functions in the following specific way.  Let $x \in \R^k$, $y_1,\ldots,y_m \in U$, and let $\Psi$ be the mapping given by
\[ (x,y_1,\ldots,y_m) \stackrel{\Psi}{\mapsto} (p_1(x,f(y)),\ldots,p_{k+md}(x,f(y))), \]
where each $p_i$ a polynomial in the $f_i(y_j)$'s and the partial derivatives of these functions. It will be necessary to count the nondegenerate multiplicity of this mapping, meaning specifically that it will be necessary for
\[ \# \set{(x,y_1,\ldots,y_m)}{ \Psi(x,y_1,\ldots,y_m) = c \mbox{ and } \left| \frac{\partial \Psi}{\partial (x,y_1,\ldots,y_m)} \right| \neq 0} \leq N \]
for some constant $N$ which is independent of $c$.  A constant $C$ depending on $f_1,\ldots,f_k$ and $U$ will be called ``topological'' when it depends on the nondegenerate multiplicity of finitely many mappings of the form $\Psi$.
Note that it suffices to assume that this condition holds for all such $\Psi$ when $x$ is restricted to the unit ball.  This fact may be established by adjoining a scaling parameter $\lambda$ to the $x$-variables so that
\[ (\lambda, x,y_1,\ldots,y_m) \stackrel{\Psi'}{\mapsto} ( p_1 (x, f(y), \lambda),\ldots, p_{md+k} (x,f(y),\lambda), \lambda) \]
and the polynomials are homogeneous
\[ {\Psi}'(x,y,\lambda) = (a_1,\ldots,a_{md+k},1) \Leftrightarrow {\Psi}'(t x, y, t \lambda) = (t^{s_1} a_1,\ldots,t^{s_{md+k}} a_{md+k}, t) \]
so that every solution $\Psi(x,y) = c$ corresponds to a solution $\Psi'(x,y,\lambda) = c'$ for which $|x|^2 + \lambda^2 \leq 1$).

All of the constants appearing in theorems \ref{geomthm} and \ref{geomthm2} as well as the lemmas \ref{codim1lemma}--\ref{codim3lemma} will be topological constants.  As a consequence of the work of Khovanski\u\i \cite{khovanskii1980}, \cite{khovanskii1991}, the number of nondegenerate solutions of any polynomial system of equations is finite and may be bounded in terms of the degrees (which is, in this case, also a consequence of Bezout's theorem).  In particular, when $C$ is a topological constant and $f_1,\ldots,f_k$ are polynomials, then $C$ is automatically finite and may be taken to depend only on the degrees and the mappings $\Psi$ involved.  Khovanski\u\i's theorem applies in the general situation of Pfaffian functions, which are real-analytic functions satisfying certain systems of differential equations on $U$ as well, but this will not play a major role here.  It should also be mentioned that when $U$ is bounded $f_1,\ldots,f_k$ are real analytic on some neighborhood of the closure of $U$, the work of Gabrielov \cite{gabrielov1968} on projections of semi-analytic sets guarantees again that topological constants are finite and uniform under analytic perturbations of the functions $f_i$; more details may be found in the work of Gabrielov and Vorobjov \cite{gv2004}.  Thus, theorems \ref{theorem1} and \ref{theorem2} will also hold when $\rho$ is real analytic on some neighborhood of the compact set $\overline{U}$, although the uniformity of the constant will hold only for fixed, small analytic perturbations of $\rho$.  Finally, we remark that the most general situation in which topological constants must be finite and uniform under perturbation is when the functions $f_i$ arise from an o-minimal structure.  The reader is referred to van den Dries \cite{vandendries1998} for a short and highly readable introduction to this general subject.

In general, the precise form of the topological constants involved in the various theorems will be too complex to explicitly identify (thanks to the induction procedure, many different systems of equations must be required to have bounded nondegenerate multiplicity).  In the case of theorems \ref{theorem1} and \ref{geomthm}, however, it is possible to explicitly identify the constant in question, at least in the case when $d_l = d_r$.  Let
\[ \Sigma_y^{L,c} := \set{ x \in \R^{d_l} } { (x,y) \in U, \ \rho(x,y) = c, \partial_x \rho(x,y) \neq 0}. \]
For theorems \ref{theorem1} and \ref{geomthm} to hold with a finite constant, it is sufficient to assume that there is a finite constant $N_L$ satisfying two constraints.  The first is that, for any $y_1,\ldots,y_{d_l} \in \R^{d_r}$ and any $c_{1},\ldots,c_{d_l} \in \rho(E)$, 
\begin{equation} 
\begin{split} 
\#  \set{x \in \Sigma_{y_{1}}^{L,c_{1}} \cap \cdots \cap \Sigma_{y_{d_l}}^{L,c_{d_l}}}{
 % x \in \Sigma_{y_i}^{R,t_i} \ \forall i, 
   \left| \partial_x \rho(x,y_1) \cdots \partial_x \rho(x,y_{d_l}) \right| \neq 0 } \leq N_L
\end{split} \label{boundconst}
\end{equation}
(throughout this paper, the notation $|v_1 \cdots v_k|$ will refer to the absolute value of the determinant of the matrix with rows $v_1,\ldots,v_k$).
In other words, any $d_l$-tuple of submanifolds $\Sigma_{y_i}^{L,c_i}$ has no more than $N_L$ points at which the hypersurfaces intersect transversely.  The second constraint on $N_L$ is that, for all $c \in \rho(E)$, the left projection $\pi_L : {\mathscr C}_c \rightarrow T^* (\R^{d_l})$ also has nondegenerate multiplicity at most $N_L$ (meaning again that there are at most $N_L$ points in ${\mathscr C}_c$ at which the Jacobian of the mapping is nonvanishing which are sent to any fixed image point).  %This constant $N_L$ is also the same constant that appears in theorem \ref{geomthm}.

\subsection{Notation and outline}

Given any vector $v \in \R^N$, the notation $v^{(k)}$ will denote the projection of $v$ onto the subspace spanned by the first $k$ coordinate directions (and likewise for $E^{(k)}$ when $E \subset \R^N$).  For any set closed $E \subset \R^{d_l} \times \R^{d_r}$, the notation $E_x$ will be used to refer to either of the sets
\[ \set{y \in \R^{d_r}}{ (x,y) \in E} \mbox{ or }  E \cap \left( \{x\} \times \R^{d_r} \right), \]
 the particular meaning being determined by context.  Likewise $E^y$ will equal either
\[ \set{x \in \R^{d_l}}{ (x,y) \in E} \mbox{ or }  E \cap \left( \R^{d_l} \times \{y\} \right). \]
All mappings will be assumed smooth unless otherwise stated (although at no point will more than a finite amount of differentiability be explicitly required).  General sets $E$ will be taken to be closed (unless explicitly stated otherwise) to ensure that any continuous image of $E$ into $\R^N$ is measurable.  In practice, all that is actually required of $E$ is that it be Lebesgue measurable along with its relevant projections (and, in fact, the standard approximation argument yields conclusions valid for any measurable $E$ so long as the measure of the projections is replaced by the outer measure). 
Finally, if $\varphi$ is any mapping defined on $U \subset \R^{d_l} \times \R^{d_r}$ and $E \subset U$ is closed, the left and right widths of $E$ with respect to $\varphi$ will refer to the quantities
\[ |\varphi(E)|_L := \mathop{\mathrm{ess.sup}}_y |\varphi(E^y)| \mbox{ and } |\varphi(E)|_R := \mathop{\mathrm{ess.sup}}_x |\varphi(E_x)|. \]

The present paper refines and improves both the geometric perspective and the inductive constructions which appeared in \cite{gressman2010II} in several fundamental ways:  first, the functions $f_i$ will not necessarily be taken on the real line, but on $\R^d$.  This simple change requires a fundamentally new approach for the proof, but the spirit of the results remains the same (since, for example, one may recover estimates for the multilinear, one-dimensional setting by testing \eqref{mainobj} on functions $f_i$ which factor into a $d_i$-fold product of functions each of which depends on only one coordinate).  The second extension of \cite{gressman2010II} to be implemented here is to shed more light on the geometric interpretation of some of the natural quantities governing the behavior of \eqref{mainobj}.  The third and final extension to be undertaken here is to establish uniform, weighted versions of \eqref{mainobj}.

%Optimal regularity in Sobolev spaces for FIOs with one-sided folds Comech \cite{comech1999}

The outline of the rest of this paper is as follows:  section \ref{prelim} is devoted to what may seem at first to be three independent but fairly elementary observations.  The first deals with multilinear determinant functionals and is essentially a minor adaptation of the main theorem of \cite{gressman2010}, which establishes the connection between the Oberlin-type rectangle condition (in its modified form appearing in theorem \ref{geomthm}) to multilinear determinant functionals.  The second topic of section \ref{prelim} is to establish something akin to a converse of theorem \ref{geomthm}, namely, that certain Radon-like integral estimates imply estimates of the form \eqref{geomassump} (this latter sort of estimate will be called a ``geometric integral'' in much of what follows).  This argument is merely an extension of the observation by Oberlin \cite{oberlin2000II} that the rectangle condition can be estimated by testing convolution operators on rectangles.  The final goal of section \ref{prelim} is to establish a version of the change-of-variables formula which will be broadly useful throughout the remainder of the paper.

Section \ref{theorempf} is devoted to the proof of theorem \ref{geomthm} and its companion theorem \ref{geomthm2}.  Throughout section \ref{theorempf}, the condition \eqref{geomassump} will be treated as a black box; this is in contrast with the related results of Oberlin \cite{oberlin2000II} which establish estimates using an explicit characterization of the weights involved (namely, that it equals a power of the Gaussian curvature).

Sections \ref{inductionsec} and \ref{specificsec} are devoted to an explicit study of the curvature condition \eqref{geomassump}.  Specifically, section \ref{inductionsec} describes an inductive procedure by which knowledge of a particular weight functional $w$ (depending on $\rho$ and $\varphi$) which establishes estimates like \eqref{geomassump} will necessarily imply the existence of a family of such weights.  These new weights are relatively easy to describe in terms of the old weights (the definition is given by \eqref{newfunc}), but the geometric interpretation of the construction is far from obvious.  For that reason, section \ref{specificsec} is devoted to the explicit calculation and analysis of the first few weights in the series generated by theorem \ref{gentheorem}.

Section \ref{finalsec} contains the proofs of theorems \ref{theorem1} and \ref{theorem2}, which are accomplished by applying the results of section \ref{specificsec} to theorem \ref{geomthm}.  Section \ref{finalsec} also contains a number of related results of interest including estimates implied by the nondegeneracy of minors of the matrix \eqref{radonrotcurv} (in the spirit of Cuccagna's result for FIOs \cite{cuccagna1996}) as well as a generalization of theorem \ref{greenleafseeger} corresponding to the situation of Greenleaf and Seeger \cite{gs1994} of an arbitrary number of nonvanishing principal curvatures.

Finally, section \ref{appendix} contains the appendix, which explains in greater detail the transition from sublevel set operator to Radon-like transform via the coarea formula.

\section{Preliminaries}
\label{prelim}
\subsection{Coercive determinant functional inequalities}

In this subsection, it is presumed that one has established an inequality of the ``geometric'' type, and this inequality is used to establish a coercive estimate for a multilinear determinant functional.  In particular, suppose $U \subset \R^d$ is open and $\Phi : U \rightarrow \R^n$ is continuous.  It will be assumed that there exists a nonnegative measurable function $w$ on $U$ so that
\begin{equation}
 \sup_{Q \in \GL(n,\R)} |Q|^{\frac{s}{n}} \int_U \frac{w(y) dy}{|Q \Phi(y)|^{s}} \leq K_{\Phi,w,s} < \infty  \label{boxes}
\end{equation}
holds for some positive $s$.  The conclusion is the following:
\begin{lemma}
Suppose that \eqref{boxes} holds.  Then there is a constant $c_n$ depending only on $n$ such that
\begin{equation*}
\begin{split} \int_{U^n} | \Phi(y_1) \cdots \Phi(y_n)| &  \prod_{j=1}^n |g_j(y_j)| w(y_j) dy_1 \cdots dy_n \\
 \geq c_n & \frac{(s+1)^n}{s^n} K_{\Phi,w,s}^{-\frac{n}{s}} \prod_{j=1}^n \left( \int_U |g_j(y)|^{\frac{s}{s+1}} w(y) dy \right)^{\frac{s+1}{s}}
\end{split} %\label{maindet}
\end{equation*}
for any measurable functions $g_1,\ldots,g_n$ on $U$. Here $K_{\Phi,w,s}$ is exactly the same constant as in \eqref{boxes}. In particular, it also must be the case that
\begin{equation}
\begin{split} \int_{U^n} | \Phi(y_1) \cdots \Phi(y_n)| &  \prod_{j=1}^n |g_j(y_j)| dy_1 \cdots dy_n \\
 \geq c_n & K_{\Phi,w,s}^{-\frac{n}{s}} \prod_{j=1}^n \left( \int_U |g_j(y)|^{\frac{s}{s+1}} |w(y)|^{\frac{1}{s+1}} dy \right)^{\frac{s+1}{s}}
\end{split} \label{maindet}
\end{equation}
for any measurable $g_1,\ldots,g_n$ on $U$ provided that $s > 0$.
\end{lemma}

\begin{proof}
This lemma is essentially a trivial consequence of theorem 3 of \cite{gressman2010}. The content of that theorem (restricted to the case of interest) is as follows:  suppose $\mu$ is a nonnegative Borel measure on $\R^n$ such that
\begin{equation} \int_B d \mu \leq C |B|^{\frac{s}{n}} \label{boxest0} \end{equation}
for all ellipsoids $B \subset \R^n$ centered at the origin.  Then
\begin{equation} \int_{E_1 \times \cdots \times E_n} |x_1 \cdots x_n| d\mu(x_1) \cdots d \mu(x_n) \geq C' \prod_{i=1}^n \mu (E_j)^{\frac{s+1}{s}} \label{boxconc}
\end{equation}
for all measurable sets $E_i \subset \R^n$, where $C'$ equals $C^{-\frac{n}{s}}$ times a constant depending only on $n$ and $s$.  The first task at hand is to show that, when $\mu$ has the special form
\[ \int g d \mu := \int_U g( \Phi(y)) w(y) dy \]
with $U \subset \R^d$ and some $\Phi : U \rightarrow \R^n$,
then the estimate \eqref{boxest0} implies the corresponding inequality
\begin{equation} \int_{U^n} \! \! \! |\Phi(y_1) \cdots \Phi(y_n)| \prod_{j=1}^n \chi_{F_j}(y_j) w(y_j) d y_1 \cdots d y_n \geq C' \prod_{i=1}^n \left[ \int_{F_j} \! \! w(y) dy \right]^{\frac{s+1}{s}}  \label{detimprov}
\end{equation}
for measurable $F_i \subset U \subset \R^d$ and the same $C'$ as above (the difference being that the sets $F_i$ are subsets of $\R^d$ while the sets $E_i$ are subsets of $\R^n$).  To that end, fix $F_1,\ldots,F_n$.  Note that the condition \eqref{boxest0} guarantees that $\mu$ is $\sigma$-finite.  By extension, the measure $\mu_j$ defined by
\[ \int g d \mu_j := \int_U g(\Phi(y)) \chi_{F_j}(y) w(y) dy \]
is also $\sigma$-finite.  Since the $\mu_j(E) \leq \mu(E)$ for any measurable $E \subset \R^n$, the Radon-Nykodym theorem guarantees that 
\[ \int g d \mu_j = \int g e_j d \mu \]
for some function $e_j$ on $\R^n$ which must be nonnegative and bounded above by $1$ $\mu$-almost everywhere. Now fix some $\epsilon$ strictly between $0$ and $1$ and let
\[ E_{j}^k := \set{ x \in \R^n}{ e_j(x) \geq \epsilon^{k+1}} \]
for $j=1,\ldots,n$ and $k$ any nonnegative integer.  Clearly the comparisons
\[ (1-\epsilon) \sum_{k=0}^\infty \epsilon^{k+1} \chi_{F_j^k}(x) \leq e_j(x) \leq (1-\epsilon) \sum_{k=0}^\infty \epsilon^k \chi_{F_j^k}(x) \]
hold $\mu$-a.e.; applying \eqref{boxconc} to the sets $E_1^{k_1},\ldots,E_n^{k_n}$, multiplying both sides by $\epsilon^{k_1+\cdots+k_n+n}(1-\epsilon)^n$, and summing over $k_1,\ldots,k_n$ gives that
\begin{equation*}
\begin{split}
 \int \! | x_1 \cdots x_n |  \prod_{j=1}^n e_j(x_j)  d \mu(x_1) \cdots d \mu(x_n) \geq C' (1-\epsilon)^n \! \! \! \! \! \! \sum_{k_1,\ldots,k_n = 0}^\infty  \prod_{j=1}^n \epsilon^{k_j+1} \mu(E_j^k)^{\frac{s+1}{s}}
\end{split} 
\end{equation*}
for any $\epsilon > 0$.
%H\"{o}lder's inequality may be used to estimate the right-hand side of \eqref{detfunc2}.  
By Jensen's inequality applied to the sum over the $k_i$'s,
\[ (1-\epsilon) \sum_{k=0}^\infty \epsilon^k \mu(E_j^k) \leq \left( (1-\epsilon) \sum_{k=0}^\infty \epsilon^k \mu(E_j^k)^{\frac{s+1}{s}} \right)^{\frac{s}{s+1}}; \]
consequently, it must be the case that
\begin{equation*}
\begin{split}
 \int | x_1 \cdots x_n |  & \prod_{j=1}^n e_j(x_j)  d \mu(x_1) \cdots d \mu(x_n)   \geq C' \epsilon^n  \prod_{j=1}^n \left( \int e_j d \mu \right)^{\frac{s+1}{s}}.
\end{split} %\label{detfunc3}
\end{equation*}
All that remains to establish \eqref{detimprov} is to let $\epsilon \rightarrow 1$ and recall the definition of $e_j$.
%% \begin{equation}
%% \begin{split}
%%  \int | x_1 \cdots x_n |  \prod_{j=1}^n \chi_{E_j}(x_j) & d \mu(x_1) \cdots d \mu(x_n) \\ & \geq c_n ( s K^{-1})^{\frac{n}{s}} \prod_{j=1}^n \left( \int_{E_j} d \mu \right)^{\frac{s+1}{s}}
%% \end{split} \label{detfunc}
%% \end{equation}
%% holds for $\mu$-measurable sets $E_1,\ldots,E_n$ in $\R^n$.  It is preferable, however, to consider sets in $U \times (0,\infty)$.  To that end, let $F_1,\ldots,F_n$ and $F_0$ be any measurable sets in $U \times (0,\infty)$ such that
%% \[ \int_{F_0} \lambda^{s-1} w(y) dy d \lambda < \infty. \]
%% By the Radon-Nykodym theorem, there are nonnegative, $\mu$-measurable functions $f_1,\ldots,f_n$ such that
%% \[ \int_{F_0} g(\lambda \Phi(y)) \chi_{F_j}(y,\lambda) \lambda^{s-1} w(y) dy d \lambda = \int_{F_0} g(\lambda \Phi(y)) f_j(\lambda \Phi(y)) \lambda^{s-1} w(y) dy d \lambda \]
%% for any nonnegative $\mu$-measurable function $g$ (let the measure $\mu$ restricted to $F_0$ be called $\mu_0$).  The functions $f_1,\ldots,f_n$ are clearly less than or equal to $1$ almost everywhere (with respect to $d \mu_0$ as well as with respect to $d y d \lambda$).  
%% Sending $\epsilon \rightarrow 1$ and letting $F_0$ exhaust $U \times (0,\lambda)$ gives that
%% \begin{align*}
%%  \int & | \lambda_1 \Phi(y_1) \cdots \lambda_n \Phi(y_n)| \prod_{j=1}^n \chi_{F_j}(y_j,\lambda_j) \lambda_j^{s-1} w(y_j) d y_1 d \lambda_1 \cdots d y_n d \lambda_n \\
%% & \geq c_n (s K^{-1})^{\frac{n}{s}} \prod_{j=1}^n \left( \int_0^\infty \! \! \int_U \chi_{F_j}(y,\lambda) \lambda^{s-1} w(y) dy d \lambda \right)^{\frac{s+1}{s}}.
%% \end{align*}

To finish the lemma, consider the measure $\mu$ on $\R^n$ given by
\[ \int g d \mu := \int_0^\infty \! \! \int_U g(\lambda \Phi(y)) \lambda^{s-1} w(y) dy d \lambda. \]
This is clearly a nonnegative Borel measure; moreover, Fubini's theorem dictates that
\[ \int \chi_{|Q x| \leq 1} d \mu(x) = \frac{1}{s} \int_U \frac{w(y) dy}{|Q \Phi(y)|^{s}} \]
so the condition \eqref{boxes} is equivalent to the assertion that, for any centered ellipsoid $B \subset \R^n$,
\[ \int_B d \mu \leq s^{-1} K |B|^{\frac{s}{n}}. \]
To conclude, fix measurable functions $g_1(y),\ldots,g_n(y)$, and apply \eqref{detimprov} to the sets $F_j$ given by $F_j := \set{(y,\lambda)}{ \lambda \leq |g_j(y)|^{1/(s+1)}}$.  By Fubini's theorem, the lemma is established, since
\begin{align*}
 \int |\lambda_1 \Phi(y_1)  \cdots \lambda_n \Phi(y_n)| & \prod_{j=1}^n \lambda_j^{s-1} \chi_{F_j}(y_j,\lambda_j) d \lambda_1 \cdots d \lambda_n \\
&  = |\Phi(y_1) \cdots \Phi(y_n)| (s+1)^{-n} \prod_{j=1}^{n} |g_j(y_j)| 
\end{align*}
and $\int \chi_{F_j}(y,\lambda) d \lambda = s^{-1} |g_j(y)|^{s}$.
\end{proof}

\subsection{Geometric integral inequalities}

Next it will be presumed that an estimate for a Radon-like operator has been obtained, and this estimate will be used to deduce an inequality of the geometric type, namely, of the form \eqref{boxes}.  To that end, let $U$, $\Phi$, and $w$ be as before, and assume that
\begin{equation} 
\left| \int_{\R^n} \int_U f(x+ \Phi(y)) g(x) w(y) dy dx \right| \leq C_{\Phi,w,r} ||f||_{L^p(\R^n)} ||g||_{L^q(\R^n)} \label{simpleradon}
\end{equation}
for some $p,q$ satisfying $\frac{1}{p} + \frac{1}{q} = \frac{2}{r}$.  Note that the change of variables $x \mapsto -x -\Phi(y)$ combined with the reflections $f(x) \mapsto f(-x)$ and $g(x) \mapsto g(-x)$ mean that \eqref{simpleradon} will also hold with the same constant when $p$ and $q$ are exchanged.  By Riesz-Thorin, then, one may assume without loss of generality that \eqref{simpleradon} holds for $p = q = r$.

\begin{lemma}
Suppose that \eqref{simpleradon} holds.  Let $(\Phi,1) : U \rightarrow \R^{n+1}$ be given by $(\Phi(y),1) := (\Phi_1(y),\ldots,\Phi_n(y),1)$.  Then
\begin{equation}
\sup_{Q \in \GL(n+1,\R)} |Q|^{\frac{2}{r}-1} \int_U \frac{w(y) dy}{|Q (\Phi(y),1)|^{(n+1)( \frac{2}{r}-1)}} \leq c_n C_{\Phi,w,r}
\end{equation}
for some constant $c_n$ depending only on $n$ and the constant $C_{\Phi,w,r}$ equals the corresponding constant in \eqref{simpleradon}. \label{srlemma}
\end{lemma}
\begin{proof}
Suppose $Q \in \GL(n+1,\R)$, and consider the quadratic form $|Q z|^2$ on $(z_1,\ldots,z_{n+1}) \in \R^{n+1}$.  Temporarily regard $z_{n+1}$ as fixed; as a function of $(z_1,\ldots,z_n)$, the polynomial $|Qz|^2$ is still degree two, nonnegative, and tends to infinity as $(z_1,\ldots,z_n) \rightarrow \infty$; completing the square and exploiting homogeneity leads one to the conclusion that there must exist $Q_0 \in \GL(n,\R)$, $v \in \R^n$, and a positive constant $b$ such that
\[ |Q z|^2 = |Q_0 (z_{\widehat{n+1}} - v z_{n+1})|^2 + b^2 z_{n+1}^2 \]
where $z_{\widehat{n+1}} := (z_1,\ldots,z_n)$.  Furthermore, the skew transformation $z_{\widehat{n+1}} \mapsto z_{\widehat{n+1}} + v z_{n+1}$ has determinant one, so it must be the case that $|Q| = |Q_0| b$.

Now fix $\tilde r := \frac{n+1}{r} - \frac{1}{2}$, and consider the following functions on $\R^n$:
\begin{align*}
f(x) & := ( |Q_0 (x-v)|^2 + b^2 )^{-\frac{\tilde r}{2}}, \\
g(x) & := ( |Q_0 x |^2 + b^2)^{-\frac{\tilde r}{2}}. 
\end{align*}
The first step is to estimate the integral
\[ \int_{\R^n} f(x + \Phi(y)) g(x) dx \]
from below; once this is accomplished, an estimate from above using \eqref{simpleradon} will complete the lemma.  To that end, suppose $|Q_0 x|^2 \leq \frac{1}{4}( |Q_0 (\Phi(y) - v)|^2 + b^2)$.  The triangle inequality gives
\begin{align*}
 |Q_0(x + \Phi(y) - v)|^2 + b^2 & \leq |Q_0(\Phi(y) - v)|^2 + b^2\\
& \hspace{20pt}  + |Q_0 x|^2 + 2 |Q_0 x| |Q_0(\Phi(y) - v)| \\
& \leq \frac{9}{4} \left( |Q_0(\Phi(y) - v)|^2 + b^2 \right);
\end{align*}
Letting $E := \set{x \in \R^n}{ |Q_0 x| \leq \frac{1}{2} \sqrt{ |Q_0(\Phi(y) - v)|^2 + b^2 }}$, it follows that
\begin{align*}
\int_{\R^n} f(x + \Phi(y)) g(x) dx \geq \left( \frac{2}{3} \right)^{\tilde r} \left( |Q_0(\Phi(y) - v)|^2 + b^2 \right)^{- \frac{\tilde r}{2}} \int_E \frac{dx}{(|Q_0 x|^2 + b^2)^{\frac{\tilde r}{2}}}.
\end{align*}
To estimate the integral on the right-hand side, change $x \mapsto Q_0^{-1} x$ and use spherical coordinates; the result is that
\[ \int_E \frac{dx}{(|Q_0 x|^2 + b^2)^{\frac{\tilde r}{2}}} = c_n |Q_0|^{-1} \int_0^{\frac{\sqrt{|Q_0(\Phi(y) - v)|^2 + b^2 }}{2}} t^{n-1} (t^2 + b^2)^{-\frac{\tilde r}{2}} dt. \]
If $t$ is further restricted so that $t \geq \frac{\sqrt{|Q_0(\Phi(y) - v)|^2 + b^2 }}{4}$, then $t^2 + b^2 \leq \frac{17}{16} t^2$.  All together, these imply that
\begin{equation}
 \int_{\R^n} f(x + \Phi(y)) g(x) dx \geq c_{n,\tilde r,r} |Q_0|^{-1} \left( |Q_0(\Phi(y) - v)|^2 + b^2 \right)^{\frac{n-2\tilde r}{2}} \label{lower1}
\end{equation}
(note that, since $1 \leq r < 2$, the constant $c_{n,\tilde r,r}$ may be assumed to only depend on $n$).
On the other hand, $||f||_r = ||g||_r$ and
\[ ||g||_r = \left( c_n |Q_0|^{-1} \int_0^\infty \frac{t^{n-1} dt}{(t^2 + b^2)^{\frac{r \tilde r}{2}}} \right)^{\frac{1}{r}} = c_{n,r,\tilde r} |Q_0|^{-\frac{1}{r}} b^{-\tilde r + \frac{n}{r}}, \]
so assuming that \eqref{simpleradon} holds, it follows that
\[ \int_U \frac{w(y) dy}{\left( |Q_0(\Phi(y) - v)|^2 + b^2 \right)^{\frac{n+1}{2} \left( \frac{2}{r} - 1\right)}} \leq c_n C_{\Phi,w,r} |Q_0| (|Q_0|^{-\frac{1}{r}} b^{-\tilde r + \frac{n}{r}})^2.\]
Since $-2 \tilde r + \frac{2n}{r} = 1 - \frac{2}{r}$, the right-hand side equals $c_n C_{\Phi,w,r} |Q|^{1-(2/r)}$; the lemma follows now from identifying the quantity being integrated on the right-hand side as the quantity $|Q (\Phi(y),1)|$ to the appropriate power.
\end{proof}

\subsection{The change-of-variables formula}

Section \ref{prelim} will now be concluded with a special application of the change-of-variables which will frequently be useful throughout the rest of the paper.

Suppose $U \subset \R^n \times \R^n$ is open.  For any fixed positive integer $m$, let
\[ U^* := \set{ (x,y_1,\ldots,y_m) \in \R^n \times (\R^{d})^m}{ (x,y_i) \in U, \ i=1,\ldots,m} \]
and consider the mapping $\Psi : U^* \rightarrow \R^{n+md}$ which is given by
\begin{align*}
 (x,y_1,\ldots,y_m) \stackrel{\Psi}{\mapsto} ( \rho_1(x,y_1),\ldots,\rho_m(x,y_m),  \eta(x,y_1),\ldots,\eta(x,y_m)) 
\end{align*}
where each $\rho_i$ is a $C^1$ map from $U$ into $\R^{k_i}$ with $\sum_{i=1}^m k_i = n$ and $\eta$ is a $C^1$ map from $U$ into $\R^{d}$.  The purpose of the following calculations is to record the various details associated to the application of the change-of-variables formula to this particular mapping $\Psi$.  To that end, consider first the task of computing the Jacobian determinant of $\Psi$.  In block form, the Jacobian matrix is given by
\[
\left| 
\begin{array}{ccccc}
\partial_x \rho_1(x,y_1) & \partial_y \rho_1(x,y_1) & 0 & \cdots & 0 \\
\partial_x \rho_2(x,y_2) & 0 & \partial_y \rho_2(x,y_2) & \cdots & 0 \\
\vdots & \vdots & \vdots & \ddots & \vdots \\
\partial_x \rho_m(x,y_m) & 0 & 0 & \cdots & \partial_y \rho_m(x,y_m) \\
\partial_x \eta(x,y_1) & \partial_y \eta(x,y_1) & 0 & \cdots & 0 \\
\partial_x \eta(x,y_2) & 0 & \partial_y \eta(x,y_2) & \cdots & 0 \\
\vdots & \vdots & \vdots & \ddots & \vdots \\
\partial_x \eta(x,y_m) & 0 & 0 & \cdots & \partial_y \eta(x,y_m)
\end{array}
\right|
\]
where $\partial_x \rho_1$, for example, is a $k_1 \times n$ matrix whose $(i,j)$-entry is exactly $\frac{\partial (\rho_1)_i}{\partial x_j}$.  In particular, $\partial_y \eta$ is a $d \times d$ matrix.  Assuming that this matrix is invertible, it is possible via column reduction to eliminate the blocks $\partial_x \eta(x,y_i)$.  Once this is accomplished, the matrix will be block upper-triangular with $m+1$ square blocks on the diagonal.  Of these, $m$ of them will be exactly the blocks $\partial_y \eta(x,y_i)$ for $i=1,\ldots,m$.  The remaining $n \times n$ square block is given by
\begin{equation} \left[ \begin{array}{c}
\partial_x \rho_1(x,y_1) - \partial_y \rho_1(x,y_1) (\partial_y \eta(x,y_1))^{-1} \partial_x \eta(x,y_1) \\
\vdots \\
\partial_x \rho_m(x,y_m) - \partial_y \rho_m(x,y_m) (\partial_y \eta(x,y_m))^{-1} \partial_x \eta(x,y_m)
\end{array}  \right]. \label{matrix1}
\end{equation}
Consequently, if one defines
\begin{equation}
 \Psi_j(x,y) := \partial_x \rho_j(x,y) - \partial_y \rho_j(x,y) (\partial_y \eta(x,y))^{-1} \partial_x \eta(x,y), \label{rows}
\end{equation}
it must be the case that
\[ \left| \frac{\partial \Psi}{\partial (x,y_1,\ldots,y_m)} \right| = \left| \Psi_1(x,y_1) \cdots \Psi_m(x,y_m) \right| \prod_{i=1}^m | \partial_y \eta (x,y_i) | \]
where $\left| \frac{\partial \Psi}{\partial (x,y_1,\ldots,y_m)} \right|$ is understood to be the row-wise concatenation as in \eqref{matrix1}.
%Assuming that the mapping $\Psi$ satisfies the finite multiplicity property, 
The change-of-variables formula now provides the following identity:
\begin{equation}
\begin{split}
  \int_{U^*} & \left| \frac{\partial \Psi}{\partial (x,y_1,\ldots,y_m)} \right|  \prod_{i=1}^m |f_i(\rho_i(x,y_i),\eta(x,y_i))| dx \cdots d y_m \\
 %& \leq N \prod_{i=1}^m ||f_i||_1 ||g_i||_1. \\
 & = \int N_{\Psi,U_*} (u,v) \prod_{i=1}^m |f_i(u_i,v_i)| du_1 \cdots du_m dv_1 \cdots dv_m,
 \end{split} \label{upper}
 \end{equation}
where $N_{\Psi,U^*}(u,v) = N_{\Psi,U^*}(u_1,\ldots,u_m,v_1,\ldots,v_m)$ equals the number of points $(x,y_1,\ldots,y_m) \in U^*$ for which $\Psi(x,y_1,\ldots,y_m) = (u,v)$ and the Jacobian determinant is nonvanishing. %{\bf IN SOME PLACES IT'S CALLED AN INEQUALITY DUE TO PREVIOUS FORM...}

It is worth noting at this point that if $n_i = 1$ for a particular $i$, then $\Psi_i$ will be a row vector.  By Cramer's rule,  the $j$-th entry of the vector $(\det \partial_y \eta) \Psi_i$ will be exactly equal to
\begin{equation}
 (\det \partial_y \eta) \frac{\partial \rho_i}{\partial x_j} +  \det \left[  \begin{array}{cc}  0 & \partial_y \rho_i \\ \frac{\partial \eta}{\partial x_j} & \partial_y \eta \end{array} \right]. \label{cramer}
\end{equation}

\section{Reduction to geometric integral inequalities}
\label{theorempf}
This section begins with the proof of theorem \ref{geomthm}, which is restated in a slightly stronger form below (using the left width $|\rho(E)|_L$ rather than $|\rho(E)|$):
\begin{theoremp}
Let $\rho : U \rightarrow \R$, where $U$ is an open subset of $\R^{d_l} \times \R^{d_r}$.  Fix a closed set $E \subset U$ and suppose that there exists a weight $w(x,y)$ for which the uniform geometric integral estimate 
\begin{equation} \sup_{x \in \R^{d_l}} \sup_{Q \in \GL(d_l,\R)} |Q|^{\frac{s}{d_l}} \int_{U_x} \frac{|w(x,y)| \chi_{E}(x,y) }{|Q \partial_x \rho(x,y)|^s} dy  \leq K_E  \label{geomassumpp}
\end{equation}
holds.  Then there is a constant $C$ such that 
%  Then when
%\[ E_{\epsilon,B} := \set{ (x,y) \in U}{ |\rho(x,y)| \leq \epsilon \mbox{ and } \varphi(x,y) \in B} \]
%and $\epsilon \leq \epsilon_0$, one has the uniform estimate 
\begin{equation} 
\int_E |w(x,y)|^{\frac{1}{s+1}} |f(x)g(y)| dx dy \leq C K_E^{\frac{1}{s+1}} |\rho(E)|_L^{\frac{s}{s+1}}  ||f||_{p'} ||g||_{\frac{s+1}{s}}  \label{geomineqp}
\end{equation}
 holds for all $f,g$, where \label{geomthmp}
\[ \frac{1}{p'} = 1 - \frac{1}{d_l} \frac{s}{s+1} \]
and the constant $C$ equals
 $C_{d_l} N_L^{\frac{s}{d_l(s+1)}}$ for some fixed $C_{d_l}$ depending only on dimension and the constant $N_L$ from \eqref{boundconst} which counts nondegenerate solutions of a system of equations on $U$ associated to the function $\rho$.
\end{theoremp}
\begin{proof}[Proof of theorem \ref{geomthm}.]
The proof begins with an application of \eqref{upper}.  The particular application relevant to the present situation will have $m = d_l$ and $d = d_r$.  Fix a single function $\rho$ and let $\rho_i := \rho$ for $i=1,\ldots,d_l$.  Fix $\eta := y$ as well.  Finally, let $F^y := \rho(E^y)$.  Clearly $\chi_{E}(x,y) \leq \chi_{F^y}(\rho(x,y))$.
Now \eqref{upper} and \eqref{cramer} imply
\begin{align*}
 \int_{U^*} & \! \! \left| \partial_x \rho(x,y_1) \cdots  \partial_x\rho(x,y_{d_l}) \right|  \prod_{i=1}^{d_l} \chi_{E}(x,y_i) |g(y_i)| dx dy_1 \cdots dy_{d_l} \\
 & \leq \int_{U^*} \! \! \left| \partial_x \rho(x,y_1) \cdots  \partial_x\rho(x,y_{d_l}) \right|  \prod_{i=1}^{d_l} \chi_{F^{y_i}}(\rho(x,y_i)) |g(y_i)| dx dy_1 \cdots dy_{d_l} \\
& \leq N_L \int \prod_{i=1}^{d_l} \chi_{F^{y_i}}(u_i) |g(y_i)| du_1 d y_1 \cdots du_{d_l} dy_{d_l} \leq N_L ( |\rho(E)|_L ||g||_1 )^{d_l}
\end{align*}
for the constant $N_L$ characterized by \eqref{boundconst}. 
% and any $\epsilon \leq \epsilon_0$ (where $\epsilon_0$ is implicitly specified in defining $N_L$). 
Assuming that there exists a weight $w(x,y)$ so that \eqref{geomassumpp}
holds, the inequality \eqref{maindet} and Fubini will give that
\begin{align*}
 & \int_{U^*}  \! \! \left| \partial_x \rho(x,y_1), \ldots, \partial_x\rho(x,y_{d_l}) \right|  \prod_{i=1}^{d_l} \chi_{E}(x,y_i) |g(y_i)| dx dy_1 \cdots dy_{d_l} \\
& \ \ \geq \int c_{d_l}( K_E)^{-\frac{d_l}{s}} \left( \int_{U_x} |g(y)|^{\frac{s}{s+1}} |w(x,y)|^{\frac{1}{s+1}} \chi_E(x,y) dy \right)^{\frac{d_l (s+1)}{s}} \! \! dx.
\end{align*}
The upper and lower bounds on the multilinear determinant functional (and the replacement of $|g|$ with $|g|^{\frac{s+1}{s}}$) combine to give that
\begin{align*}
 \int \left( \int_{U_x} |g(y)| |w(x,y)|^{\frac{1}{s+1}} \chi_{E}(x,y) dy \right)^{\frac{d_l (s+1)}{s}} & dx \\
 \leq C_{d_l} K_E^{\frac{d_l}{s}} N_L |\rho(E)|^{d_l} & \left( \int_{U_x} |g(y)|^{\frac{s+1}{s}} dy \right)^{d_l}. 
\end{align*}
Taking both sides to the power $\frac{s}{d_l (s+1)}$ gives the theorem by duality.
\end{proof}

Theorem \ref{geomthm} has a natural extension to the setting where the geometric integral inequality \eqref{geomassumpp} might fail to hold in general but nevertheless holds when $\partial_x \rho(x,y)$ is projected onto some lower-dimensional subspace.  This corresponds to the situation in which, for example, the Phong-Stein rotational curvature vanishes but the rotational curvature matrix nevertheless has some nondegenerate minors.

The first step is to consider a suitable definition for the projection operators which will be applied to $\partial_x \rho(x,y)$.  Let $V$ consist of the vectors $v_1,\ldots,v_{k+1} \in \R^d$.  One may assign to this collection a linear mapping $T_V$ which sends $\R^d$ into $\R^{d-k-1}$ such that
\begin{equation}
 | T_V z_1  \cdots  T_V z_{d-k-1} | = |v_1 \cdots  v_{k+1} \ z_1 \cdots z_{d-k-1}| \label{mapping}
\end{equation}
for any $z_1,\ldots,z_{d-k-1} \in \R^d$.  To accomplish this, simply choose linear functionals $\ell_1,\ldots,\ell_{d-k-1}$ which are linearly independent and annihilate $v_1,\ldots,v_{k+1}$, and let $T_V z = (\ell_1(z),\ldots,\ell_{d-k-1}(z))$.  Note that both sides of \eqref{mapping} depend only on the equivalence classes $z_i + {\mathrm{span}}\{v_1,\ldots,v_k\}$ and are (aside from the absolute values) alternating.  Thus, by the uniqueness of the determinant on $\R^{d-k-1}$, it is possible to rescale $T_V$ (as it has already been defined) by a constant to ensure that \eqref{mapping} holds.  Moreover, note that if $T_{V}$ and $T_{V}'$ both satisfy \eqref{mapping}, then $T_{V} = Q T_{V}'$ for some linear transformation $Q$ with determinant $\pm 1$.  In particular, the quantity
\[ \sup_{Q \in \GL(d-k-1,\R)} |Q|^{\frac{s}{d-k-1}} \int_U \frac{w(y)}{1+|Q T_V \Phi(y)|^{s}} dy \]
is independent of the particular choice of $T_V$.

%In what follows, it is also convenient to let $B_\delta$ represent the product of intervals
%\[ B_\delta := [-\delta_1,\delta_1] \times \cdots \times [-\delta_N,\delta_N] \]
%for any finite sequence $\delta_1,\ldots,\delta_N$.

With this definition, the appropriate generalization of theorem \ref{geomthm} may be stated as follows:
\begin{theorem}
Fix $k \geq 0$, $U \subset \R^{d_l} \times \R^{d_r}$, and smooth mappings $\rho : U \rightarrow \R$, $\varphi : U \rightarrow \R^N$ for $N \geq k$.  Let $V_{x,y} := ( \partial_x \varphi_1(x,y),\ldots,\partial_x \varphi_{k}(x,y), \partial_x \rho(x,y))$.  Suppose that there exists a weight $w(x,y)$ such that
\begin{equation} \mathop{\sup_{Q \in \GL(d_l-k-1,\R)}}_{(x,y_1) \in U} |Q|^{\frac{s}{d_l-k-1}} \int_U \frac{w(x,y) \chi_E(x,y)}{1 + |Q T_{V_{x,y_1}} \partial_x \rho(x,y)|^{\alpha}} dy \leq K_E \label{geomasmp2}
\end{equation}
for some real, fixed $\alpha$.  Let $\overline{s} := \frac{d_l-k}{d_l-k-1}s$.  Then there is some topological constant $C$ such that
\begin{equation} 
\int \! \chi_{E}\! (x,y) |w(x,y)|^{\frac{1}{\overline{s}+1}} |f(x)g(y)| dx dy \leq C  {K'_E}^{\frac{1}{\overline{s}+1}} |\rho(E)|_L^{\frac{\overline{s}}{\overline{s}+1}} %(\delta_1 \cdots \delta_k K_\delta)^{\frac{1}{s+1}} 
||f||_{p'} ||g||_{\frac{\overline{s}+1}{\overline{s}}}  \label{geomineq2}
\end{equation}
 for all $f,g$, where \label{geomthm2}
\[ \frac{1}{p'} = 1 - \frac{1}{d_l-k} \frac{\overline{s}}{\overline{s}+1}, \]
and $K'_E := (|\varphi^{(k)}(E)|_L)^{\overline{s}/(d_l-k)} K_E$.
\end{theorem}
Notice that the condition \eqref{geomasmp2} differs from \eqref{geomassumpp} in two significant ways.  The first is that the parameter $\alpha$ is essentially irrelevant.  Thanks to the constant term in the denominator, the condition \eqref{geomasmp2} is a closer analogue of the Oberlin condition than is \eqref{geomassumpp} because it will be true for some $\alpha$ whenever one can uniformly control the weighted measures of the sets $\set{y}{ |Q T_{x,y_1} \partial_x \rho(x,y)| \leq 1}$.  On the other hand, the condition \eqref{geomasmp2} is slightly less general than \eqref{geomassumpp} in the sense that there is a supremum over $y_1$ as well as a supremum over $x$.  This is definitely undesirable in general (since the weight $w(x,y)$ is not allowed to depend on $y_1$, the only obvious way to find such a weight is to treat $y_1$ as a parameter and take the infimum), but it is not clear if the dependence on $y_1$ can be avoided in all cases.  It is worth noting, however, that when the cutoff functions $\varphi$ are taken to depend linearly on the coordinate functions and $\rho$ is semi-translation-invariant (for example) it is possible to arrange it so that the dependence of $T_{V_{x,y_1}} \partial_x \rho(x,y)$ on $y_1$ is essentially trivial.
\begin{proof}
Once again, the proof is based on \eqref{upper}; in this case, one will take $m := d_l - k$. Let $\rho_1(x,y) := (\rho(x,y), \varphi^{(k)}(x,y))$ (where $\varphi^{(k)}$ refers to the first $k$ coordinate directions of $\varphi$) and $\rho_i := \rho$ for $i=2,\ldots,d_l-k$.  Finally, $\eta_i := y$ again for $i=1,\ldots,d_l-k$.  As in the proof of theorem \ref{geomthm}, one may use the inequality $\chi_E(x,y_j) \leq \chi_{\rho(E^{y_j})}(\rho(x,y_j))$ for $j=2,\ldots,d_l-k$.  When $j=1$, on the other hand, a more elaborate inequality will be used.
In this case,  $\chi_E(x,y_1) \leq \chi_{\rho(E^{y_1})}(\rho(x,y_1)) \chi_{\varphi^{(k)}(E^y)}(\varphi^{(k)}(x,y))$ is used. 
By \eqref{upper}, as in theorem \ref{geomthm}, one has
\begin{align*}
 \int_{U^*} \! \!  \left| \frac{\partial \Psi}{\partial(x,y_1,\ldots,y_{d_l-k})} \right|  \prod_{j=1}^k \chi_{E}(x,y_j) & |g(y_i)| dx dy_1 \cdots dy_{d_l-k} \\
& \leq C |\varphi^{(k)}(E)|_L |\rho(E)|_L^{d_l-k} ||g||_1^{d_l-k},
\end{align*}
where
\[ \left| \frac{\partial \Psi}{\partial(x,y_1,\ldots,y_{d_l-k})} \right| = \left| \partial_x \varphi_1(x,y_1) \cdots \partial_x \varphi_k(x,y_1) \ \partial_x \rho(x,y_1) \cdots  \partial_x\rho(x,y_{d_l-k}) \right|\]
and $C$ is a topological constant depending on $\rho$, $\varphi$ and $U$.  %% Let $V_{x,y}$ be the vector space spanned by $\partial_x \varphi_1(x,y),\ldots, \partial_x \varphi_k(x,y)$ and $\partial_x \rho(x,y)$.  Clearly the alternating map
%% \[ (z_1,\ldots,z_{d-k-1} ) \mapsto \left| \partial_x \varphi_1(x,y_1) \cdots \partial_x \varphi_k(x,y_1) \ \partial_x \rho(x,y_1) \ z_1  \cdots  z_{d_l-k-1} \right|\]
%% depends only on the equivalence classes of $z_1,\ldots,z_{d_l-k-1}$ modulo $V_{x,y_1}$.
By the definition of $T_{V_{x,y_1}}$, the Jacobian determinant may be expressed as a determinant with columns $T_{V_{x,y_1}} \partial_x \rho(x,y_2)$ through $T_{V_{x,y_1}} \partial_x \rho(x,y_{d_l-k})$. The geometric integral assumption \eqref{geomasmp2} implies that measures of ellipsoids $|Q T_{V_{x,y_1}} \partial_x \rho(x,y)| \leq 1$ are bounded above by $K_E |\rho(E)| |Q|^{s/(d_l - k -1)}$.  Thus \eqref{detimprov} applies and one has
\begin{align*}
 \int_{U^*_{x,y_1}} \! \!  \left| \frac{\partial \Psi}{\partial(x,y_1,\ldots,y_{d_l-k})} \right|  \prod_{i=2}^{d_l-k}  \chi_{E_i}(x,y_j) w(x,y_i) dy_2 \cdots dy_{d_l-k} \\
\geq c (K_E)^{-\frac{d_l - k - 1}{s}} \prod_{i=2}^{d_l-k} \left( \int_{U_x} \chi_{E_i}(x,y) |w(x,y)| dy \right)^{\frac{s+1}{s}}
\end{align*}
for any subsets $E_i \subset E$
(here, of course, $U^*_{x,y_1}$ is meant to be the subset of $U^*$ on which $x$ and $y_1$ are fixed).  Multiplying both sides by $w(x,y_1)$ and integrating with respect to $y_1$ gives
\begin{align*}
\int_{U^*_x} \! \!  & \left| \frac{\partial \Psi}{\partial(x,y_1,\ldots,y_{d_l-k})} \right| \prod_{i=1}^{d_l-k} \chi_{E_{i}}(x,y_i) |w(x,y_i)|  dy_1 \cdots dy_{d_l-k} \\
& \hspace{20pt} \geq c (K_E |\rho(E)|)^{-\frac{d_l - k - 1}{s}} \left( \int_{U_x} \chi_{E_1}(x,y) |w(x,y)| dy_1 \right) \\
 & \hspace{50pt} \times \prod_{i=2}^{d_l-k} \left( \int_{U_x} \chi_{E_i}(x,y) |w(x,y)| dy \right)^{\frac{s+1}{s}}.
\end{align*}
Note that this inequality has a natural interpretation in terms of weighted Lorentz spaces: the characteristic function $\chi_{E_i}(x,y)$ may be replaced with $\chi_{E_i}(x,y) |g_i(y)|$ and the integrals on the right-hand side may be replaced with the $L^1$ norm of $g_1$ with respect to the weight $\chi_E(x,y) |w(x,y)|$ and the weighted $L^{s/(s+1),1}$ Lorentz-space norms of the remaining functions.  By the standard convexity properties of the Lorentz spaces, then, it follows that
\begin{align*}
\int_{U^*_x} \! \!  & \left| \frac{\partial \Psi}{\partial(x,y_1,\ldots,y_{d_l-k})} \right| \prod_{i=1}^{d_l-k} \chi_{E}(x,y_i) |g(y_i)| |w(x,y_i)|  dy_1 \cdots dy_{d_l-k} \\
& \hspace{20pt} \geq c (K_E |\rho(E)|)^{-\frac{d_l - k - 1}{s}} \left( \int_{U_x} \chi_{E} (x,y) |g(y)|^{\frac{\overline{s}}{\overline{s}+1}} |w(x,y)| dy_1 \right)^{\frac{(\overline{s}+1)(d_l-k)}{\overline{s}}}
\end{align*}
where $\overline{s} := \frac{d_l-k}{d_l-k-1} s$.  From here, the rest of \eqref{geomineq2} follows from the same algebraic manipulations used to establish \eqref{maindet} and \eqref{geomineq}.
%% \begin{align*}
%%  (1-\tau) \sum_{k \geq k_0} \tau^{\alpha k} b_k & \leq \tau^{-k_0(1-\alpha)} (1-\tau) \sum_{k \geq k_0} \tau^k b_k \\
%% (1-\tau) \sum_{k \leq k_0} \tau^{\alpha k} b_k & \leq \left( (1-\tau) \sum_{k \leq k_0} \tau^{k (\alpha(s+1)-s)} \right)^{\frac{1}{s+1}} \left( (1 - \tau) \sum_{k \leq k_0} \tau^k b_k^{\frac{s+1}{s}} \right)^{\frac{s}{s+1}}
%% \end{align*}
\end{proof}

\section{Analysis of geometric integral inequalities}
\label{inductionsec}

With the proofs of theorems \ref{geomthm} and \ref{geomthm2} complete, it suffices to restrict attention to the Oberlin-type curvature assumptions \eqref{geomassumpp} and \eqref{geomasmp2}.  The ultimate goal of this section is the proof of theorem \ref{gentheorem}, which allows one to build new weights out of old weights via an inductive procedure.  This is perhaps the most unexpected feature of the Oberlin-type conditions \eqref{geomassumpp} and \eqref{geomasmp2}, namely, that there are many different weight functionals which satisfy uniform estimates with only topological constants.  The downside of the inductive procedure is that the weight functionals are given only implicitly, and the dependence of the next functional on the previous functional is somewhat complex.  For this reason, section \ref{specificsec} is devoted to a series of explicit calculations giving the first few weight functionals in the series.

In what follows, it will be assumed that $\varphi$ is a mapping from $U$ into $\R^N$ for some sufficiently large $N$.  Given a set $E$ contained in $U$, it will be convenient to define $|\varphi(E)|$ as a sequence whose $j$-th entry equals $|\varphi^{(j)}(E)|,$ i.e., the $j$-dimensional Lebesgue measure of the projection of $\varphi(E)$ onto the first $j$ coordinate directions.

\begin{theorem}
For fixed $d,n$, suppose \label{gentheorem} that there is a weight functional $W_y$ such that %let $\gamma$ be some multiindex of length $N$ with nonnegative entries such that $|\gamma| = 1$ and $\gamma_i = 0$ for $i<d-n+1$.   Suppose that the inequality
\begin{equation}
 \sup_{Q \in \GL(n,\R)} |Q|^{\frac{\alpha}{n}} \int_E \frac{|W_{y} (\Phi,\varphi)|^{\beta} }{| Q \Phi(y)|^{\alpha}}  dy \leq C |\varphi(E)|^{\gamma} \label{indhyp1}
\end{equation}
holds for all smooth $\Phi : U \rightarrow \R^n$ and $\varphi : U \rightarrow \R^N$, and all closed $E \subset U \subset \R^{d+1}$ where $\alpha,\beta$ are nonnegative, $\gamma$ is some multiindex of length $N \geq d-n+2$ with nonnegative entries such that $|\gamma| = 1$ and $\gamma_i = 0$ for $i<d-n+2$,
 and $C$ is a topological constant depending on $\Phi,\varphi$ and $U$.  The weight $W_y$ will also be assumed to satisfy the property that for each $\Phi,\varphi$, $W_y(\Phi,\varphi)$ is a lower semi-continuous function of $y$.

Under the assumptions above, there exists a weight functional $W'_y$ (given by \eqref{newfunc}) defined on smooth functions $\Phi : U \rightarrow \R^{n+1}$ and $\varphi : U \rightarrow \R^{N-1}$ for $U \subset \R^{d}$ such that
\begin{equation*}  
\sup_{Q \in \GL(n+1,\R)} |Q|^{\frac{(n-1) \alpha}{n(\alpha+1)}} \int_{E} \frac{ |W'_{y} (\Phi, \varphi) |^{\frac{\beta}{\alpha+1}}}{ |Q \Phi(y)|^{\frac{(n^2-1) \alpha}{n(\alpha+1)}}} dy \leq C' |\varphi({E})|^{\gamma''} %= C' \delta |\tilde \varphi(\tilde{E})|^{\gamma''}
 \end{equation*}
for some new multiindex $\gamma''$ and a topological constant $C'$.  The new weight $W'_y$ is also lower semi-continuous in $y$ for any fixed, smooth $\Phi$ and $\varphi$.  The multiindex $\gamma''$ is obtained by taking the mean $\frac{\gamma}{\alpha+1} + \frac{\delta_{d-n+2} \alpha}{\alpha + 1}$ (where $\delta_{d-n+2}$ is the multiindex equal to $1$ in the $d-n+2$ place and zero elsewhere) and removing the first entry.
\end{theorem}

Given a smooth mapping $\Phi$ from $U \subset \R^{d+1}$ into $\R^{n}$ (where $d + 2\geq n$) and a smooth mapping $\varphi$ from $U \rightarrow \R^{d-n+2}$, consider the following derivative map defined by duality in terms of the following determinant:
\begin{equation} \ang{D_{1} \Phi(y),z} := \det \left[ 
\begin{array}{cc}
0 & \partial_y \varphi(y) \\
z &\partial_y \Phi(y)
\end{array}
\right] \label{rotd1}
\end{equation}
where, as before $\partial_y \varphi$ is a $(d-n+2)\times (d+1)$ matrix of partial derivatives, $z$ is a column vector of length $n$, and so on.
If $\R^{n}$ is considered only to be an $n$-dimensional, oriented, real vector space, then $D_{1} \Phi$ may be immediately identified as an element of the dual space (i.e., as a row vector in this case).  The operator $D_1$ satisfies two major transformation laws of interest:  the first is that $D_1 [Q \Phi] = (\det Q) (Q^T)^{-1} D_1 \Phi$ for any $Q \in \GL(n,\R)$, and the second is that $D_1 [\Phi \circ \rho] = \left( \det \frac{\partial \rho}{\partial y} \right) ( D_1 [ \Phi] ) \circ \rho$.

The proof of theorem \ref{gentheorem} is achieved by closing a loop which begins with lemma \ref{srlemma}.  The key missing piece is the following lemma, which is something of a converse to lemma \ref{srlemma}:
\begin{lemma}
For fixed $d,n$, let $\gamma$ be some multiindex of length $N$ with nonnegative entries such that $|\gamma| = 1$ and $\gamma_i = 0$ for $i<d-n+2$.   Suppose that the inequality \eqref{indhyp1}
%% \begin{equation}
%%  \sup_{Q \in \GL(n,\R)} |Q|^{\frac{\alpha}{n}} \int_E \frac{  \chi_{B_\delta}(\varphi(y)) }{| Q \Phi(y)|^{\alpha}} |W_{y} (\Phi,\varphi)|^{\beta} dy \leq C |\varphi(E)|^{\gamma} \label{indhyp1}
%% \end{equation}
holds for all closed $E \subset U \subset \R^{d+1}$ where $\alpha,\beta$ are nonnegative %, $\gamma$ is supported away from the first $d-n+1$ indices, 
 and $C$ is a topological constant depending on $\Phi,\varphi$ and $U$.  Then for any measurable functions $f,g$ on $\R^n$, the inequality
\[ \left| \int_{\R^n} \! \! \int_E f(x + \Phi(y)) g(x) |W_y (D_1 \Phi, \varphi)|^{\frac{\beta}{\alpha+1}} dy dx \right| \leq C' |\varphi(E)|^{\gamma'} ||f||_{p} ||g||_{q} \]
holds uniformly for all $E$, where $C'$ is a new topological constant depending on $\Phi, \varphi$ and $U$.  The exponents $p,q$ satisfy
\[ \frac{1}{p} + \frac{1}{q} = 1 + \frac{n-1}{n} \frac{\alpha}{\alpha+1}, \]
and $\gamma' = \frac{\gamma}{\alpha + 1} + \frac{ \alpha \delta_{d-n+2}}{\alpha + 1}$, where $\delta_{d-n+2}$ is the multiindex which equals $1$ at the location $d-n+1$ and zero elsewhere.  In particular, $|\gamma'| = 1$, the entries of $\gamma'$ are nonnegative, and $\gamma'_i = 0$ for $i < d - n + 2$. \label{indhalf1}
\end{lemma}
\begin{proof}
The proof is essentially a reprise of the arguments used to establish theorem \ref{geomthm}.  
Let $\eta(x,y) := (\varphi_2(y),\ldots,\varphi_{d-n+2}(y), x + \Phi(y))$ and $\rho_i(x,y) := \varphi_1(y)$ for $i=1,\ldots,n$ %; the inequality \eqref{upper} holds subject to the finite multiplicity property; in this case, the reduction to a compact set of parameters is natural---since there is translation invariance, the multiplicity of any solution must equal the multiplicity of some other solution where $x=0$.
The rows \eqref{cramer} are most easily expressed as linear functionals; in particular, they are given exactly by $D_1 \Phi$, so  
$\left|\frac{\partial \Psi}{\partial (x,y_1,\ldots,y_n)} \right| = \left| D_{1}\Phi(y_1) \cdots D_{1}\Phi(y_n) \right|$ and \eqref{upper} dictates that 
\begin{align*} 
\int_{\R^n \times E^n}  \left|\frac{\partial \Psi}{\partial (x,\ldots,y_n)} \right| \prod_{i=1}^n |f(x + \Phi(y_i))|  %\chi_{B_{\delta}}(\varphi(y_i)) & 
dx d y_1 \! \! \cdots d y_n  \leq N ||f||_1^n |\varphi^{(d-n+2)}(E)|^{n}
%\prod_{j=1}^{d-n+1} \delta_j^{n}
\end{align*}
for some topological constant $N$.
By \eqref{maindet}, the left-hand side may be bounded below using the estimate \eqref{indhyp1} applied to $D_1 \Phi$ rather than to $\Phi$ itself (and note that the structure of $D_1$ guarantees that topological constants depending on $D_1 \Phi$ are also topological constants depending on $\Phi$).
The result is that
\begin{align*}  \int \left( \int_E \right. |f(x+\Phi(y))|^{\frac{\alpha}{\alpha+1}} & \left| W_y (D_1 \Phi, \varphi) \right|^{\frac{\beta}{\alpha+1}} dy \left. \vphantom{\int} \right)^{\frac{n(\alpha+1)}{\alpha}} dx  \\ \leq  C' &  |\varphi(E)|^{\frac{n}{\alpha} \gamma} \left( ||f||_1 |\varphi^{(d-n+2)}(E)| \right)^{n}.
\end{align*} 
Replacing $|f|$ with $|f|^{\frac{\alpha+1}{\alpha}}$ and simplifying gives the conclusion of the lemma.
\end{proof}

Combining this lemma with lemma \ref{srlemma} immediately leads to the conclusion that the condition \eqref{indhyp1} implies another condition of exactly the same form with $\Phi$ replaced with $(\Phi,1)$, which maps into $\R^{n+1}$.  With a little extra work, however, one can establish the same inequality for mappings into $\R^{n+1}$ which does not make any assumption about the form, hence completing the loop.

\begin{proof}[Proof of theorem \ref{gentheorem}.]
Consider now the situation where $U \subset \R^d$ and $\Phi : U \rightarrow \R^{n+1}$.  Augment $y \in U \subset \R^d$ with an additional parameter $\lambda$ so that $(\lambda,y) \in \R \times U$, and consider
\begin{align*}
\tilde \Phi(\lambda,y) & := \lambda \hat \Phi(y) := \lambda ( \Phi_2(y),\ldots,\Phi_{n+1}(y)) , \\
\tilde \varphi(\lambda, y) & := ( \lambda \Phi_1(y)-1, \varphi_1(y),\varphi_2(y),\ldots).
\end{align*}
The pair $\tilde \Phi, \tilde \varphi$ is defined on a subset of $\R^{d+1}$ and $\tilde \Phi$ maps into $\R^n$, so the assumption \eqref{indhyp1} applies directly to $\tilde \Phi$ and $\tilde \varphi$; thus the conclusion of lemma \ref{indhalf1} applies with $\tilde \Phi, \tilde \varphi$ in place of $\Phi$ and $\varphi$, respectively.  Finally, lemma \ref{srlemma} holds by virtue of the fact that the conclusion of lemma \ref{indhalf1} is the hypothesis of lemma \ref{srlemma}.
  Consequently, it follows from the assumption \eqref{indhyp1} that
%Assume that \eqref{indhyp1} holds for mappings from $U_0 \subset \R^{d+1}$ into $\R^{n}$; this will imply a correspnding inequality for mappings from $U \subset \R^{d}$ into $\R^{n+1}$ as follows: apply the previous lemma along with lemma \ref{srlemma} to the new mappings $(\tilde \Phi, \tilde \varphi)$ (note that this pair will have finite topological constants if the original pair does as well).  Consequently, the induction hypothesis \eqref{indhyp1} (namely, that \eqref{indhyp1} holds uniformly for some topological constant $C$) implies that
\[ \sup_{Q \in \GL(n+1,\R)} \! \! |Q|^{\frac{(n-1) \alpha}{n(\alpha+1)}} \int_{\tilde E} \frac{ |W_{\lambda,y} ({D}_1 (\tilde \Phi), \tilde \varphi) |^{\frac{\beta}{\alpha+1}}}{ |Q (1,\tilde \Phi)|^{\frac{(n^2-1) \alpha}{n(\alpha+1)}}} dy d \lambda \leq C' |\tilde \varphi(\tilde E)|^{\gamma'}  \]
holds uniformly for some topological constant $C'$, every $Q \in \GL(n+1,\R)$, and every closed $\tilde E \subset \R \times U$.  Note that the notation $W_{\lambda,y}$ refers exactly to the weight function of \eqref{indhyp1} but emphasizes that the first coordinate direction of integration, $\lambda$, will be analyzed separately from the rest, $y$.
%Note that, at the moment, $I$ has $(d+1)-n+1$ initial entries which equal $1$; likewise, the first $(d+1)-n+1$ entries of $\gamma$ vanish.  
The multiindex $\gamma'$ has nonnegative entries which vanish for $i < d - n + 2$ and sum to $1$.
Note that, by the definition \eqref{rotd1} and the definitions of $\tilde \Phi$ and $\tilde \varphi$, one has
%${\tilde D}_1$ is meant to indicate the use of \eqref{rotd1} in the augmented coordinates; as a block-form determinant, 
\begin{equation}
 \ang{ {D}_1 \tilde \Phi, z} = (-1)^{d-n+1} \lambda^{n-1} \det \left[   
\begin{array}{ccc}
0 & 0 & \partial_y \varphi^{(d-n+1)}  \\
0 & \Phi_1 &\partial_y \Phi_1  \\
z & \hat \Phi & \partial_y \hat \Phi
\end{array} \label{dminus}
\right] \end{equation}
where $\partial_y \varphi^{(d-n+1)}$ is meant to consist of only $d-n+1$ rows (necessary to make the matrix square).  
Let $D_{-} \Phi$ be defined using the determinant on the right-hand side of \eqref{dminus} so that $\ang{{\tilde D}_1 \tilde \Phi, z} = (-1)^{d-n+1} \lambda^{n-1} \ang{D_{-} \Phi, z}$ for all $z \in \R^{n}$.  Thus, assuming that \eqref{indhyp1} holds, then
\[  |Q|^{\frac{(n-1) \alpha}{n(\alpha+1)}} \int_{\tilde E} \frac{ |W_{\lambda,y} (\lambda^{n-1} D_{-} \Phi, \tilde \varphi) |^{\frac{\beta}{\alpha+1}}}{ |Q (1,\tilde \Phi)|^{\frac{(n^2-1) \alpha}{n(\alpha+1)}}} dy d \lambda \leq C' |\tilde \varphi(\tilde{E})|^{\gamma'} \]
for some topological constant $C'$.  To eliminate $\lambda$, fix any closed $E \subset U$ and let $\tilde E_\delta := \set{ (\lambda,y)}{ y \in E \mbox{ and } |\lambda \Phi_1(y) - 1| \leq \frac{\delta}{2}}$.  Clearly one has for every $j$ that $|\tilde \varphi(\tilde E_\delta)|_j \leq \delta |\varphi(E)|_{j-1}$, so if $\gamma''$ equals the multiindex $\gamma'$ with the first index truncated, it follows that$|\tilde \varphi(\tilde E_\delta)|^{\gamma'} \leq \delta |\varphi(E)|^{\gamma''}$ because $|\gamma'| = 1$.  Thus
\begin{equation}  |Q|^{\frac{(n-1) \alpha}{n(\alpha+1)}} \int_{\tilde E_\delta } \frac{ |W_{\lambda,y} (\lambda^{n-1} D_{-} \Phi, \tilde \varphi) |^{\frac{\beta}{\alpha+1}}}{ |Q (1,\tilde \Phi)|^{\frac{(n^2-1) \alpha}{n(\alpha+1)}}} dy d \lambda \leq C' \delta |\varphi({E})|^{\gamma''}. %= C' \delta |\tilde \varphi(\tilde{E})|^{\gamma''}
\label{fatou0} \end{equation}
Note that, by assumption on $W_{\lambda,y}$, the function
\[ \frac{ |W_{\lambda,y} (\lambda^{n-1} D_{-} \Phi, \tilde \varphi) |^{\frac{\beta}{\alpha+1}}}{ |Q (1,\tilde \Phi)|^{\frac{(n^2-1) \alpha}{n(\alpha+1)}}} \] 
is lower semi-continuous as a function of $\lambda$ (in the event that the ratio is undefined, simply set it equal to zero at that point).  Thus it follows that, at any point $y$, the following inequality must hold:
\begin{align*}
 \liminf_{\delta \rightarrow 0^+} \frac{1}{\delta} \int \chi_{|\lambda \Phi_1(y) - 1| \leq \frac{\delta}{2}} & \frac{ |W_{\lambda,y} (\lambda^{n-1} D_{-} \Phi, \tilde \varphi) |^{\frac{\beta}{\alpha+1}}}{ |Q (1,\tilde \Phi)|^{\frac{(n^2-1) \alpha}{n(\alpha+1)}}} d \lambda \\
 \geq & \frac{1}{|\Phi_1(y)|}\left. \frac{ |W_{\lambda,y} (\lambda^{n-1} D_{-} \Phi, \tilde \varphi) |^{\frac{\beta}{\alpha+1}}}{ |Q (1,\tilde \Phi)|^{\frac{(n^2-1) \alpha}{n(\alpha+1)}}} \right|_{\lambda = \frac{1}{\Phi_1(y)}} \\
 = & \frac{|\Phi_1(y)|^{\frac{(n^2-1) \alpha}{n(\alpha+1)} - 1}}{|Q \Phi(y)|^{\frac{(n^2-1) \alpha}{n(\alpha+1)}}} \left. |W_{\lambda,y} (\lambda^{n-1} D_{-} \Phi, \tilde \varphi) |^{\frac{\beta}{\alpha+1}} \right|_{\lambda = \frac{1}{\Phi_1(y)}}.
\end{align*}
 Therefore if one defines
\begin{equation}
W'_y(\Phi,\varphi) :=  \left( |\Phi_1(y)|^{\frac{(n^2-1) \alpha}{n(\alpha+1)}-1}   \left. |W_{\lambda,y} (\lambda^{n-1} D_{-} \Phi, \tilde \varphi) |^{\frac{\beta}{\alpha+1}} \right|_{\lambda = \frac{1}{\Phi_1(y)}} \right)^{\frac{\alpha+1}{\beta}} \label{newfunc}
\end{equation}
(setting $W'_y(\Phi,\varphi) = 0$ if $|\Phi_1(y)| = 0$), dividing \eqref{fatou0} by $\delta$ and applying Fatou's lemma to the integral in $y$ guarantees that
\begin{equation*}  |Q|^{\frac{(n-1) \alpha}{n(\alpha+1)}} \int_{E} \frac{ |W'_{y} (\Phi, \varphi) |^{\frac{\beta}{\alpha+1}}}{ |Q \Phi(y)|^{\frac{(n^2-1) \alpha}{n(\alpha+1)}}} dy \leq C'  |\varphi({E})|^{\gamma''} %= C' \delta |\tilde \varphi(\tilde{E})|^{\gamma''}
 \end{equation*}
exactly as desired.  Note also that $W'_y(\Phi,\varphi)$ is lower semi-continuous as well.
\end{proof}
\section{Special geometric integral inequalities}
\label{specificsec}

This section is devoted to an explicit analysis of the consequences of theorem \ref{gentheorem}.  It begins with an essentially trivial ``base case'' for the induction argument and then derives several related weights.  The cases are labeled by ``codimension,'' which, in this case, refers to the largest possible value of $n-d$ to which each lemma applies ($n-d$ is the codimension of the image of $\Phi$).

\subsection{Codimension 1}

This first lemma establishes the geometric integral estimate necessary for the proof of theorem \ref{theorem1}.  It is essentially an application of the change-of-variables formula.
\begin{lemma}
Suppose $d \geq n -1$ and that $\Phi : U \rightarrow \R^n$ and $\varphi : U \rightarrow \R^{d-n+1}$ are both $C^1$-mappings. There there exists a topological constant $C$ (depending on $\Phi$, $\varphi$, and $U$) such that \label{codim1lemma}
\begin{equation}
\sup_{Q \in \GL(n,\R)} |Q| \int_E \frac{\left| \ang{D_1 \Phi,\Phi}(y) \right|}{ |Q \Phi(y)|^{n}} dy \leq C | \varphi^{(d-n+1)}(E)| \label{boxest1}
\end{equation}
for any closed $E \subset U$, where $D_1$ is as in \eqref{rotd1}.
% and $B$ is any measurable set.
\end{lemma}
Geometrically, $\left| \ang{D_1 \Phi,\Phi}(y_0) \right| \neq 0$ if and only if $\partial_y \varphi(y_0)$ has maximal rank (i.e., $d-n+1$) at $y_0$ and the mapping $y \mapsto \frac{\Phi}{|\Phi|}$ from $U$ into $S^{n-1}$ is invertible at $y_0$ when restricted to the submanifold where $\varphi(y) = \varphi(y_0)$.  
Notice also that the weight appearing in \eqref{boxest1} satisfies the transformation laws
\begin{align*}
 \left| \ang{ D_1 Q \Phi, Q \Phi} \right| & = |Q| |\ang{D_1 \Phi, \Phi}|, \\
\left| \ang{ D_1 (\Phi \circ \phi), \Phi \circ \phi} \right| & = |\ang{D_1 \Phi, \Phi}|  \left| \frac{\partial \phi}{\partial y} \right|, 
\end{align*}
meaning that it is a geometric invariant (i.e., the weight times the measure $dy$ is invariant under reparametrization) and as well as invariant under the action of $Q \in SL(n,\R)$. 
\begin{proof}[Proof of lemma \ref{codim1lemma}.]
It suffices to assume that $\varphi$ maps into $\R^{d-n+1}$ and that $E := U \cap \varphi^{-1} (B)$ for some measurable $B$.
Let $x := \lambda \in \R$, $\eta(\lambda,y) := (\varphi_2(y),\ldots,\varphi_{d-n+1}(y),\lambda \Phi(y))$ and $\rho_1 := \varphi_{1}(y)$ and use \eqref{upper} with $m=1$.  In this case, note that the topological requirement is equivalent to the requirement that, on any submanifold where $\varphi$ is constant, the mapping $\frac{\Phi}{|\Phi|}$ has bounded nondegenerate multiplicity.  Let $C_{\Phi,\varphi,U}$ equal the maximal multiplicity.
%
 %subject to $\fmp_{(\Phi,\varphi)}$.  In particular, note that it suffices to check finite nondegenerate multiplicity for $\lambda \in [-1,1]$ (since if $(\varphi(y),\lambda \Phi(y)) = (a,b)$ has $N+1$ nondegenerate solutions in the region $\lambda \in [-L,L]$ then $(\varphi(y),\lambda \Phi(y) = (a, L^{-1} b)$ must have $N+1$ nondegenerate solutions in $[-1,1]$).  
The absolute value of the determinant \eqref{cramer} is quickly verified to equal  $|\lambda|^{n-1} \left| \ang{D_{1} \Phi(y),\Phi(y)} \right|$.  Fix any $Q \in \GL(n,\R)$.  By \eqref{upper}, then,
\[ \int_{\R} \int_U |\lambda|^{n-1} \left| \ang{D_{1} \Phi(y),\Phi(y)} \right|  \chi_{|\lambda Q \Phi(y)| \leq 1} \chi_{B}(\varphi(y)) dy d \lambda \leq \frac{C_{\Phi,\varphi,U}}{ |Q|} |B|. \]
Interchanging the order of integration by Fubini allows for an explicit computation of the $\lambda$ integral 
\[ \int_{\R}|\lambda|^{n-1} \chi_{|\lambda Q \Phi(y)| \leq 1} d \lambda =  \int_{-|Q \Phi(y)|^{-1}}^{|Q \Phi(y)|^{-1}} |\lambda|^{n-1} d \lambda = \frac{2}{n} |Q \Phi(y)|^{-n}. \]
Substituting this result gives precisely \eqref{boxest1}.
\end{proof}

Lemma \ref{codim1lemma} has a natural companion corresponding to the degenerate situation anticipated by theorem \ref{geomthm2}. Just as the combination of theorem \ref{geomthm} and lemma \ref{codim1lemma} establishes uniform $L^p$--$L^q$ estimates when the rotational curvature matrix is nondegenerate (i.e., establishes theorem \ref{theorem1}), it will be possible to combine the following lemma with theorem \ref{geomthm2} when the rotational curvature matrix is assumed to have at least some minimal rank at every point.  This corresponds to the situation of Cuccagna \cite{cuccagna1996}, who proved local $L^p$-Sobolev estimates for Fourier integral operators under the same sort of hypotheses.  More details will be given in section \ref{finalsec}.

\begin{lemma}
Suppose $U \subset \R^d$ and $\Phi : U \rightarrow \R^n$ is a smooth mapping.  Fix $j \geq 0$ and let $V := \{v_1,\ldots,v_j\}$ for vectors $v_i \in \R^n$, and suppose $\varphi : U \rightarrow \R^{N}$ for $N \geq d-n+j$.  Let \label{codim1jlemma}
\[
\widetilde{W}^1_y(\Phi,\varphi) :=  \left| \begin{array}{cccc}
0 & \cdots & 0 & \partial_y \varphi^{(d-n+j)}(y) \\
v_1 & \cdots & v_k  & \partial_{y} \Phi(y) 
\end{array} \right|. \label{codim1partial}
\] 
Then for some topological constant $C$, 
\begin{equation}
\sup_{Q \in \GL(n-j,\R)} |Q| \int_E \frac{\widetilde{W}^1_y(\Phi,\varphi)}{ 1+|Q T_V \Phi(y)|^{n-j+1}} dy \leq C | \varphi^{(d-n+k)}(E)| 
\end{equation}
holds uniformly for all $V$ and $E$.  
\end{lemma} 
\begin{proof}
This lemma follows immediately from its predecessor.  Fix $T_V$ as was used in theorem \ref{geomthm2}, let $\tilde \Phi(y) := (1, T_{V} \Phi(y))$, $\tilde \varphi := \varphi$, $\tilde{n} := n-j+1$, and $\tilde d := d$.  Now apply lemma \ref{codim1lemma} to all the tilde-endowed parameters and mappings.  By the definition of $D_1$, it follows that
\[ \ang{ D_1 \tilde \Phi, \tilde \Phi} = \det \left[ \begin{array}{cc}
0 & \partial_y \tilde \varphi^{(\tilde d - \tilde n + 1)} \\
1 & 0 \\
T_v \Phi & \partial_y T_v \Phi 
\end{array} \right]. \]
On the other hand, using the definition of $T_V$ and eliminating the middle row gives
\[ \left| \ang{ D_1 \tilde \Phi, \tilde \Phi} \right| = \left| \begin{array}{cccc}
0 & \cdots & 0 & \partial_y \varphi^{(d - n + j)} \\
v_1 & \cdots & v_k & \Phi 
\end{array} \right| = \widetilde{W}^1_y(\Phi,\varphi).
\]
Thus lemma \ref{codim1lemma} implies that
\begin{equation}
\sup_{\tilde Q \in \GL(\tilde n,\R)} |\tilde Q| \int_E \frac{\widetilde{W}^1_y(\Phi,\varphi)}{ |\tilde Q \tilde \Phi(y)|^{\tilde{n}}} dy \leq C | \varphi^{(d-n+j)}(E)| 
\end{equation}
holds uniformly for some topological constant $C$.  Restricting attention to matrices $\tilde Q$ such that $\tilde Q (1,w) = (1,Q w)$ for all $w \in \R^{n-j}$ and noting that $|\tilde Q \tilde \Phi|^{n-j+1} \approx 1 + |T_V \Phi|^{n-j+1}$ completes the lemma.
\end{proof}

%% It is worth noting that the expression $\ang{D_1 \Phi, D_1^2 \Phi}$ may be expressed in the following, more familiar form:
%% \begin{equation}
%% \left| \ang{ D_1 \Phi, D_1^2 \Phi } \right| := \left|
%% \begin{array}{cccc}
%%  0 & \frac{\partial \varphi}{\partial y_1} & \cdots & \frac{\partial \varphi}{\partial y_d} \\
%% \frac{\partial \varphi^T}{\partial y_1} & \ang{D_1 \Phi, \frac{\partial^2 \Phi}{\partial y_1 \partial y_1}} & \cdots & \ang{D_1 \Phi, \frac{\partial^2 \Phi}{\partial y_1 \partial y_{d}}} \\
%% \vdots & \vdots & \ddots & \vdots \\
%% \frac{\partial \varphi^T}{\partial y_d} & \ang{D_1 \Phi, \frac{\partial^2 \Phi}{\partial y_d \partial y_1}} & \cdots & \ang{D_1 \Phi, \frac{\partial^2 \Phi}{\partial y_d \partial y_{d}}}
%% \end{array}
%% \right| \label{calc1}
%% \end{equation}
%% (here $\varphi^T$ means, of course, that $\varphi$ is expressed as a row vector).  There are two main observations necessary to establish this equality.  The first is contained in the following proposition:

\subsection{Codimension 2}
This section is devoted to the explicit computation of the weight constructed by the combination of lemma \ref{codim1lemma} and theorem \ref{gentheorem}.  Already in this first nontrivial situation, the calculations are lengthy.  The following lemma will greatly simplify the work ahead:
\begin{lemma}
Let $\ell_1$ be a linear functional on $\R^n$ given by
\[ \ang{\ell_1,z} := \det \left[ \begin{array}{cccc} 0 & B_1 & \cdots & B_d \\ z & y_1 & \cdots & y_d \end{array} \right] \]
for any $z \in \R^n$ where $y_1,\ldots,y_d \in \R^n$ are column vectors and $B_1,\ldots,B_d$ are columns of length $d-n+1$.
If $\ell_2,\ldots,\ell_k$ are linear functionals on $\R^n$ expressed as row vectors and $A_2,\ldots,A_k$ are row vectors in $\R^{k-n}$, then
\begin{equation} \det  \left[ \begin{array}{cc}
0 & \ell_1 \\
A_2 & \ell_2 \\
\vdots & \vdots \\
A_k & \ell_k
\end{array}
\right] = (-1)^{k-n} \det  \left[ \begin{array}{cccc} 
0  & B_1 & \cdots & B_d \\
A_2 & \ang{\ell_2, y_1} & \cdots & \ang{\ell_2, y_d} \\
\vdots & \vdots & \ddots & \vdots \\
A_k & \ang{\ell_k, y_1} & \cdots & \ang{\ell_k, y_d}
\end{array} \right]. \label{detred}
\end{equation}
%(where $\frac{\partial \varphi}{\partial y_j}$ is a column vector of length $d-n+1$).
\end{lemma}
\begin{proof}
Without loss of generality, one may assume that the dimension of the span of $B_1,\ldots,B_d$ is $d-n+1$ (if not, $\ell_1 = 0$ and the right-hand side of \eqref{detred} is clearly zero as well).  In this case, column reduction allows one to find linearly-independent vectors $V_1,\ldots,V_{n-1}$ in $\R^d$ such that
\[\ang{\ell_1,z} = \det[ z, V_1 y,\ldots, V_{n-1} y ] \]
for any $z$, where $V_i y := \sum_{j=1}^d (V_i)_j y_j$, and that $V_i B = 0$ for $i=1,\ldots,n-1$.
  The first step is to apply a transformation to the final $n$ columns of determinant on the left-hand side of \eqref{detred} equivalent to changing to the basis $z, V_1 y,\ldots, V_{n-1} y$ of $\R^n$ for an appropriate choice of $z$ (if no such $z$ exists so that this is a basis, then both sides of \eqref{detred} must again vanish because $\ell_1 = 0$ and the right-hand side may be column reduced to produce a vanishing column).  Expressing each $\ell_i$ in coordinates $(\ang{\ell_i,z}, \ang{\ell_i,V_1 y}, \ldots, \ang{\ell_i,V_{n-1} y})$ gives
\begin{align*}
 \det( z, & V_1 y ,\ldots, V_{n-1} y) 
\det  \left[ \begin{array}{cc}
0 & \ell_1 \\
A_2 & \ell_2 \\
\vdots & \vdots \\
A_k & \ell_k
\end{array}
 \right] \\
 & = \det  \left[ 
\begin{array}{ccccc} 
 0  & \ang{\ell_1,z} & \ang{\ell_1, V_1 y} & \cdots & \ang{\ell_1,V_{n-1} y} \\
A_2 & \ang{\ell_2, z} & \ang{\ell_2, V_1 y} & \cdots & \ang{\ell_2, V_{n-1} y} \\
 \vdots & \vdots & \vdots & \ddots & \vdots \\
A_k & \ang{\ell_k, z} & \ang{\ell_k, V_1 y} & \cdots & \ang{\ell_k, V_{n-1} y}
 \end{array} \right].
\end{align*}
Moreover, $\ang{\ell_1, V_j y} = 0$ for $j=1,\ldots,n-1$ since the determinant defining $\ell_1$ will have two identical columns.  Thus, the first row of the matrix above has only one nonzero entry, i.e., $\ang{\ell_1,z}$.  Expanding in this entry and canceling it on both sides gives that
\begin{align*}
\det  \left[ \begin{array}{cc}
0 & \ell_1 \\
A_2 & \ell_2 \\
\vdots & \vdots \\
A_k & \ell_k
\end{array}
 \right] = (-1)^{k-n} \det  \left[ 
\begin{array}{cccc} 
A_2 &  \ang{\ell_2, V_1 y} & \cdots & \ang{\ell_2, V_{n-1} y} \\
 \vdots &  \vdots & \ddots & \vdots \\
A_k & \ang{\ell_k, V_1 y} & \cdots & \ang{\ell_k, V_{n-1} y}
 \end{array} \right].
\end{align*}
The equality \eqref{detred} now follows by reversing the column reduction which gave rise to $V_1,\ldots,V_{n-1}$.
\end{proof}
%% To establish \eqref{calc1}, one invokes the definition of $D_1$ and takes the transpose, namely that
%% \[ \ang{D_1 \Phi, D_1^2 \Phi} = \det \left[ \begin{array}{cc}
%% 0 & D_1 \Phi \\
%% \frac{\partial \varphi^T}{\partial y_1} & \frac{\partial D_1 \Phi}{\partial y_1} \\
%% \vdots & \vdots \\
%% \frac{\partial \varphi^T}{\partial y_d} & \frac{\partial D_1 \Phi}{\partial y_d} 
%% \end{array} \right]. \]
%% To get from this equality to \eqref{calc1}, one needs only invoke \eqref{detred} and observe that $\ang{ Y_j D_1 \Phi, Y_k \Phi} = - \ang{D_1 \Phi, Y_j Y_k \Phi}$ for the vectors $Y_1,\ldots,Y_{n-1}$ as defined in the proposition above (since the determinant will have identical rows unless the derivative $Y_j$ falls on the column $Y_k \Phi$ in the definition of $D_1 \Phi$).  In particular, since $D_1 \Phi$ anihilates $Y_k \Phi$ for all $k$, $\ang{D_1 \Phi, Y_j Y_k \Phi} = \ang{D_1 \Phi, Y_k Y_j \Phi}$ and there is no problem changing back to the standard basis $\frac{\partial}{\partial y_i}$.

The combination of theorem \ref{gentheorem} and lemma \ref{codim1lemma} imply the existence of a new weight functional, which will be called $W^2$.
This is most easily expressed as the determinant of a matrix whose entries are themselves determinants:
\begin{align}
 W^2_y (\Phi, \varphi) & := \det \left[ \begin{array}{cc} 0 & \partial_y \varphi^{(d-n+2)} \\ (\partial_y \varphi^{(d-n+2)})^T & M \end{array}  \right] \label{weight2pt1} \\
M_{ij} & := \det \left[ \begin{array}{ccc} 0 & 0 & \partial_y \varphi^{(d-n+2)} \\
\frac{\partial^2 \Phi}{\partial y_i \partial y_j} & \Phi & \partial_y \Phi \end{array} \right]. \label{weight2pt2}
\end{align}
%where $\partial_y \varphi$ is a $(d - n + 2) \times d$ matrix whose columns are the derivatives of $\varphi$ in the specified direction, etc., and $M_{ij}$ is the formula for the $(i,j)$-entry of the $d \times d$ matrix $M$.
Geometrically, $W^2$ is nonzero at $y_0$ precisely when $\partial_y \varphi(y_0)$ has maximal rank ($d-n+2$) and the cone $\lambda \Phi(y), \lambda \neq 0$ has $n-2$ nonvanishing principal curvatures at $y_0$.  The combination of theorem \ref{geomthm} and \eqref{boxest2} is precisely what gives theorem \ref{theorem2}.  Likewise, theorem \ref{geomthm2} together with suitable variants of lemma \ref{codim2lemma} will combine to give estimates in the situation of varying ranks of the rotational curvature matrix together with varying numbers of nonvanishing principal curvatures.
\begin{lemma}
Suppose $d \geq n-2$ and $\varphi$ maps into $\R^N$ for $N \geq d-n+2$.  There is a weight $W^2$ such that, for some topological constant $C$, one has that \label{codim2lemma}
\begin{equation}
\sup_{Q \in \GL(n,\R)} |Q|^{\frac{n-2}{n}} \int_E \frac{\left| W^2_y (\Phi,\varphi) (y) \right|^{\frac{1}{n}}}{ |Q \Phi(y)|^{n-2}} dy \leq C |\varphi^{(d-n+2)}(E)| \label{boxest2}
\end{equation}
uniformly for smooth $\Phi : U \rightarrow \R^n$, $\varphi : U \rightarrow \R^N$, and $E \subset U \subset \R^d$. % with $U \subset \R^d$ and $N \geq d-n+2$.
%% \begin{equation}
%% \begin{split}
%% \left| \int_{\R^{n}} \int_U \right. f(x + \Phi(y)) g(x) \left| R_2 D_1 \Phi(y) \right|^{\frac{1}{n(n-1)}} & \prod_{i=1}^{d-n+1} \chi_{|\varphi_i(y)| \leq \delta_i} dy dx \left. \vphantom{\int} \! \right| \\
%%  \leq  C   ||f||_{\frac{n}{n-1}} &  ||g||_{\frac{n}{n-1}} \delta_{d-n+2}^{\frac{1}{n-1}} \prod_{i=1}^{d-n+1} \delta_i . \label{radonest2}
%% \end{split}
%% \end{equation}
\end{lemma}
\begin{proof}[Proof of lemma \ref{codim2lemma}.]
The estimate \eqref{boxest2} will follow immediately from theorem \ref{gentheorem} as soon as it can be established that $W^2$ is the weight functional foretold by that theorem.  Furthermore, the concerned reader may note that the hypotheses of this lemma are that $d \geq n-2$, while lemma \ref{codim1lemma} applies when $d \geq n-1$; in fact, the establishment of this lemma for minimal choice of $d$  follows immediately by augmenting $y$ to become $(y,z)$ for $z \in \R^{M}$ for some large $M$ and adding the coordinate functions of $z$ to the cutoff mapping $\varphi$.  Since $\Phi$ does not depend on $z$, it is easy to check that the weight $W^{2}_{y,z}$ reduces exactly to $W^2_y$ in this case, and the estimate \eqref{boxest2} will hold by extending $E \subset U \subset \R^d$ to $E \times [0,1]^M$.

By virtue of \eqref{newfunc}, the next weight is obtained by examining the expression $W^{1}_{\lambda,y} (\lambda^{n-1} D_{-} \Phi, \tilde \varphi)$, where $W^1_y$
%% \[ W^1_y(\Phi,\varphi) := \det \left[ \begin{array}{cc}  0 & \partial_y \varphi^{(d-n+1)} \\
%% \Phi & \partial_y \Phi \end{array} \right]
%% \]
is the weight appearing in \eqref{boxest1}. %(where here the first $d-n+1$ entries of $\varphi$ are used when $\Phi$ maps into $\R^n$).
For simplicity, assume from this point forward that $N = d-n+1$ so that the superscript on $\varphi^{(d-n+1)}$ may be taken implicitly.
  Observe that $W^1 (f \Phi,\varphi) = f^{n} W^1 (\Phi,\varphi)$ for any continuously differentiable function $f$, so that the factor $\lambda^{n-1}$ may be pulled out to become $\lambda^{n(n-1)}$, i.e.,  $W^{1}_{\lambda,y} (\lambda^{n-1} D_{-} \Phi, \tilde \varphi) = \lambda^{n(n-1)} W^{1}_{\lambda,y} (D_{-} \Phi, \tilde \varphi)$. Now the determinant $W^1_{\lambda,y}( D_{-} \Phi, \tilde \varphi)$ has the property that its second column (corresponding to differentiation with respect to $\lambda$) has only one non-zero entry (corresponding to differentiation of $\tilde \varphi_1 = \lambda \Phi_1 - 1$).  Expanding the determinant in this second column gives
\begin{align*}
 |W^1_{\lambda,y} ( \lambda^{n-1} D_{-} \Phi,\tilde \varphi)| = |\lambda^{n(n-1)} \Phi_1(y) W^1_{y} (D_{-} \Phi, \varphi)|.
\end{align*}
Next, one may expand $W^1_{y} (D_{-} \Phi,\varphi)$ according to \eqref{detred} and the definition of $D_{-} \Phi$.  In particular, $W^1_{y} (D_{-} \Phi,\varphi)$ equals the determinant of the following matrix (up to a sign):
\begin{equation}
\left[ \begin{array}{ccccc}
0 & 0 & \frac{\partial \varphi}{\partial y_1} & \cdots & \frac{\partial \varphi}{\partial y_d} \\
0 & \Phi_1 & \frac{\partial \Phi_1}{\partial y_1} & \cdots & \frac{\partial \Phi_1}{\partial y_d} \\
\frac{\partial \varphi^T}{\partial y_1} & \ang{ \frac{\partial}{\partial y_1} D_{-} \Phi, \hat \Phi} & \ang{ \frac{\partial}{\partial y_1} D_{-} \Phi, \frac{\partial}{\partial y_1} \hat \Phi} & \cdots & \ang{ \frac{\partial}{\partial y_1} D_{-} \Phi, \frac{\partial}{\partial y_d} \hat \Phi}  \\
\vdots & \vdots & \vdots & \ddots & \vdots \\
\frac{\partial \varphi^T}{\partial y_d} & \ang{ \frac{\partial}{\partial y_d} D_{-} \Phi, \hat \Phi} & \ang{ \frac{\partial}{\partial y_d} D_{-} \Phi, \frac{\partial}{\partial y_1} \hat \Phi} & \cdots & \ang{ \frac{\partial}{\partial y_d} D_{-} \Phi, \frac{\partial}{\partial y_d} \hat \Phi}
\end{array} \right].
\label{detthing}
\end{equation}
%(where, in all instances of $\varphi$, only the first $d-n+1$ entries of $\varphi$ are used). 
Multiply the row
\[ \left[ \begin{array}{ccccc} 0 & \Phi_1 & \frac{\partial \Phi_1}{\partial y_1} & \cdots & \frac{\partial \Phi_1}{\partial y_d} \end{array} \right] \]
by 
\[ \frac{\partial}{\partial y_i} \det \left[ \begin{array}{ccc}
0 & 0 & \partial_y \varphi \\
1 & \Phi_1 & \partial_y \Phi_1 \\
0 & \hat \Phi & \partial_y \hat \Phi
\end{array} \right] \]
and add the result to the lower rows (use the corresponding derivatives in $y$ for the corresponding lower rows).  This operation does not change the determinant; the result of this manipulation gives a matrix of exactly the same form as \eqref{detthing} whose lowest rows are given by 
\[ \left[ \begin{array}{ccccc} \frac{\partial \varphi^T}{\partial y_i} & \ang{ \frac{\partial}{\partial y_i} D \Phi, \Phi} & \ang{ \frac{\partial}{\partial y_i} D \Phi, \frac{\partial}{\partial y_1} \Phi} & \cdots & \ang{ \frac{\partial}{\partial y_i} D \Phi, \frac{\partial}{\partial y_d} \Phi} \end{array} \right] \]
for $i=1,\ldots,d$, where 
where
\begin{equation}
\ang{D \Phi,z} := \det \left[ \begin{array}{ccc} 0 & 0 & \partial_y \varphi \\ z & \Phi & \partial_y \Phi \end{array} \right]. \label{rotd3}
\end{equation}
Next, note that
\[ 0 = \det \left[ \begin{array}{ccc}
0 & 0 & \partial_y \varphi \\
\Phi & \Phi & \partial_y \Phi
\end{array} \right] = \ang{D \Phi, \Phi}%%  + \Phi_1 \det \left[ \begin{array}{ccc}
%% 0 & 0 & \partial_y \varphi_y \\
%% 1 & \Phi_1 & \partial_y \Phi_1 \\
%% 0 & \hat \Phi &  \partial_y \hat \Phi
%% \end{array} \right]
\]
and that, for any directional derivative $Y$ which annihilates $\varphi$,
\[ 0 = \det \left[ \begin{array}{ccc}
0 & 0 & \partial_y \varphi \\
Y_i \Phi & \Phi & \partial_y \Phi
\end{array} \right] = \ang{D \Phi, Y_i \Phi}. %% + Y_i \Phi_1 \det \left[ \begin{array}{ccc}
%% 0 & 0 & \partial_y \varphi \\
%% 1 & \Phi_1 & \partial_y \Phi_1 \\
%% 0 & \hat \Phi & \partial_y \hat \Phi
%% \end{array} \right]
%% .
\]
Differentiating the former equality with respect to $Y_i$ and subtracting the latter gives $\ang{Y_i D \Phi, \Phi} = 0$,
%% \[ 0 =  \ang{Y_i D \Phi, \Phi} + \Phi_1 Y_i \det \left[ \begin{array}{ccc}
%% 0 & 0 & \partial_y \varphi \\
%% 1 & \Phi_1 & \partial_y \Phi_1 \\
%% 0 & \hat \Phi & \partial_y \hat \Phi
%% \end{array} \right], \]
which demonstrates that the vector \[\left( \ang{\frac{\partial}{\partial y_1} D \Phi, \Phi}, \ldots, \ang{\frac{\partial}{\partial y_d} D \Phi, \Phi} \right)\] must lie in the span of the gradients of the $\varphi$'s, so it is possible by row reduction to eliminate all entries of the second column except for $\Phi_1$.  So again the determinant may be expanded in this trivial column, and the conclusion is that $|W^1_{\lambda,y} (\lambda^{n-1} D_{-} \Phi, \tilde \varphi)|$ equals 
\[ |\lambda^{n(n-1)} (\Phi_1(y))^2| 
\left[ \begin{array}{cccc}
0  & \frac{\partial \varphi}{\partial y_1} & \cdots & \frac{\partial \varphi}{\partial y_d} \\
\frac{\partial \varphi^T}{\partial y_1} & \ang{ \frac{\partial}{\partial y_1} D \Phi, \frac{\partial}{\partial y_1}  \Phi} & \cdots & \ang{ \frac{\partial}{\partial y_1} D \Phi, \frac{\partial}{\partial y_d} \Phi}  \\
\vdots &  \vdots & \ddots & \vdots \\
\frac{\partial \varphi^T}{\partial y_d}  & \ang{ \frac{\partial}{\partial y_d} D \Phi, \frac{\partial}{\partial y_1}  \Phi} & \cdots & \ang{ \frac{\partial}{\partial y_d} D \Phi, \frac{\partial}{\partial y_d}  \Phi}
\end{array} \right]. \]
  %The price paid is to replace a general term $\ang{ \frac{\partial}{\partial y_i} D_{-} \Phi, \frac{\partial}{\partial y_j} \hat \Phi}$ with $\ang{ \frac{\partial}{\partial y_i} D \Phi, \frac{\partial}{\partial y_j}  \Phi}$
Finally, note that $\ang{ Y_i D \Phi, Y_j \Phi} =  \ang{ D \Phi, Y_i Y_j \Phi}$.  To see this, express $\ang{ Y_i D \Phi, Y_j \Phi}$ as a sum of determinants in which the derivative $Y_i$ falls on the various columns (excluding the first) appearing in \eqref{rotd3}.  Unless the derivative $Y_i$ happens to fall on the column of the determinant \eqref{rotd3} corresponding to the derivative in the $Y_j$ direction, the determinant will vanish. so by antisymmetry, $\ang{ Y_i D \Phi, Y_j \Phi} = - \ang{ D \Phi, Y_i Y_j \Phi}$.

The result of all these calculations is that
\[ |W^1_{\lambda,y} (\lambda^{n-1} D_{-} \Phi, \tilde \varphi)| = |\lambda^{n(n-1)} (\Phi_1(y))^2| |W^2_y (\Phi,\varphi)|. \]
Substituting into \eqref{newfunc} gives that the new weight $W'$ will equal
\[ |W'| = |W^2_y (\Phi,\varphi)| |\Phi_1(y)|^{-n(n-1)+2} (|\Phi_1(y)|^{\frac{(n^2-1)n}{n(n+1)} - 1} )^{\frac{n+1}{1}} = |W^2_y (\Phi,\varphi)|. \]
This new weight, of course, applies when $\Phi$ maps into $\R^{n+1}$.  Shifting the index $n$ down to $n-1$ establishes \eqref{boxest2} by virtue of theorem \ref{gentheorem}.
\end{proof}

%Geometrically, the weight $W^2_y(\Phi,\varphi)$ is nonvanishing precisely when the cones
%\[ \set{ \lambda \Phi(y) \in \R^{n}}{ \lambda \neq 0, \varphi_i(y) = c_i \ i=1,\ldots,d-n+2} \]
%have $n-2$ nonvanishing principal curvatures at a particular point (which is the maximum number which may be simultaneously nonzero).  
In analogy with the work of Greenleaf and Seeger \cite{gs1994}, \cite{gs1998}, it is possible to generalize the inequality \eqref{boxest2} to the situation where only $n-k+2$ of the curvatures are nonvanishing;  to preserve uniformity, though, it is necessary to assume that the cones have only a finite extent in the non-curved directions (excluding, of course, the distinguished radial direction).  The following lemma makes this idea precise. 
\begin{lemma}
Suppose $\Phi : U \rightarrow \R^{n-1}$ and $\varphi : U \rightarrow \R^{d-n+2}$ are $C^2$ functions.  Fix an integer $k$ between $0$ and $n-2$ (inclusive) and consider the weight $\widetilde{W}^2$ given by \label{codim2klemma}
\begin{align}
 {\widetilde{W}}^2_y (\Phi, \varphi) & := \det \left[ \begin{array}{ccc} 0 & 0 & \partial_y \varphi  \\
0 & 0 & \partial_y \Phi^{(k)} \\
(\partial_y \varphi)^T & (\partial_y \Phi^{(k)})^T & \tilde M \end{array}  \right], \label{weight2kpt1} \\
{\tilde M}_{ij} & := \det \left[ \begin{array}{ccc} 0  & \partial_y \varphi \\
\frac{\partial^2 \Phi}{\partial y_i \partial y_j} & \partial_y \Phi \end{array} \right]. \label{weight2kpt2}
\end{align}
  Then
\begin{equation}
\begin{split}
  \sup_{Q \in \GL(n,\R)} \! |Q|^{\frac{n-k-2}{n-k}} \! \int_E  \frac{|\widetilde{W}^2_y(\Phi,\varphi)|^{\frac{1}{n-k}} dy }{|Q(\Phi(y),1)|^{\frac{n(n-k-2)}{n-k}}}    
\leq C' |\Phi^{(k)}(E)|^{\frac{2}{n-k}}  & |\varphi(E)| \end{split}
\label{weight2k}
\end{equation}
for some topological constant $C'$, uniformly in $\Phi$,$\varphi$, and $E$.
\end{lemma}
\begin{proof}
For fixed $\Phi, \varphi$ and $k$, construct new mappings as follows:
\begin{align*}
 \tilde \Phi_L(y) & := (1, \Phi_{k+1}(y),\ldots,\Phi_{n-1}(y)) - L (\Phi_1(y),\ldots,\Phi_{k}(y)), \\
\tilde \varphi(y) & := (\varphi_1(y),\ldots,\varphi_{d-n+2}, \Phi_1(y),\ldots,\Phi_{k}(y)),
\end{align*}
where $L$ is any linear mapping from $\R^k$ into $\R^n$.
Applying \eqref{weight2pt1} and \eqref{weight2pt2} to $(\tilde \Phi_L,\tilde \varphi)$ gives that ${W}^2_y (\tilde \Phi_L, \tilde \varphi) = {\widetilde{W}}_y^2(\Phi,\varphi)$.
%% \begin{align*}
%%  {W}^2_y (\tilde \Phi_L, \tilde \varphi) & := \det \left[ \begin{array}{ccc} 0 & 0 & \partial_y \varphi  \\
%% 0 & 0 & \partial_y \Phi_{1\ldots,k} \\
%% (\partial_y \varphi)^T & (\partial_y \Phi_{1\ldots,k})^T & \tilde M \end{array}  \right] %\label{weight2kpt1} \\
%% {\tilde M}_{ij} & := \det \left[ \begin{array}{ccc} 0  & \partial_y \varphi \\
%% \frac{\partial^2 \Phi}{\partial y_i \partial y_j} & \partial_y \Phi \end{array} \right], %\label{weight2kpt2}
%% \end{align*}
%%where $\Phi_{1,\ldots,k} := (\Phi_1(y),\ldots,\Phi_{k}(y))$. 
 To see this, note first that for any $M_{ij}$< the row corresponding to the first coordinate of $\tilde \Phi_L$ has only one nonzero entry; this entry appears  
in the second column (corresponding to $\Phi$ in \eqref{weight2pt2}).  Thus we may expand and consequently remove this row and this column.
%containing $\Phi$ in \eqref{weight2pt2} is absent in \eqref{weight2kpt2} since there will be a row of $\tilde M$ with only one nonzero entry due to the presence of the constant first coordinate of $\tilde \Phi_L$ (and that nonzero entry appears in the column in question).
Next,  in each $M_{ij}$, one may assume that the second partials of $\Phi^{(k)}$ vanish by virtue of the first derivatives $\partial_y \Phi^{(k)}$ appearing in \eqref{weight2pt1} (meaning that $y$ derivatives are only taken in directions in which $\Phi^{(k)}$ is constant). 
Altogether, this means that the remaining first derivatives of the term $L \Phi^{(k)}$ may be eliminated by row reduction.  In particular, this is why $W^2_y(\tilde \Phi_L,\tilde \varphi)$ does not depend on $L$.  

Consequently, for any matrix $Q_1 \in \GL(n-k,\R)$, one has 
\begin{equation} |Q_1|^{\frac{n-k-2}{n-k}} \int_{E} \frac{ |{\widetilde{W}}^2_y(\Phi,\varphi)|^{\frac{1}{n-k}} }{|Q_1 \tilde \Phi_L(y)|^{n-k-2}} dy \leq C |\tilde \varphi^{(d-n+2+k)}(E)| \label{fancy1}
\end{equation}
uniformly for some topological constant $C$ (which may be taken to be independent of $L$).  Observe that $|\tilde \varphi^{(d-n+2+k)}(E)| \leq |\varphi^{d-n+2}(E)||\Phi^{(k)}(E)|$.  Now fix some matrix $Q_2 \in \GL(k,\R)$ and let 
\[ D_R := \set{ y \in U}{|Q_1 \tilde \Phi_L(y)|^{2} + |Q_2 \Phi^{(k)} (y)|^2 \leq r^2}. \]
Applying \eqref{fancy1} gives the inequality
\begin{align*}
 \int_E & |{\widetilde{W}}^2_y(\Phi,\varphi)|^{\frac{1}{n-k}} \chi_{D_R}(y) dy \\
& \leq C |Q_1|^{- \frac{n-k-2}{n-k}} |\varphi^{(d-n+2)}(E)| r^{n-k-2} \min \{ r^k |Q_2|^{-1} , |\Phi^{(k)}(E)| \} 
\end{align*}
(where the factor of $r^{n-k-2}$ comes from bounding the denominator and the two alternatives for the minimum arise from the fact that $\Phi^{(k)}(E \cap D_R)$ must be at most comparable to $r^k |Q_2|^{-1}$).  Multiplying both sides by $r^{-\frac{n}{n-k} (n-k-2)}$ and summing dyadically over $r$ gives
\[ \begin{split}  \int_E  \frac{|{\widetilde{W}}^2_y(\Phi, \varphi)|^{\frac{1}{n-k}}}{(|Q_1 \tilde \Phi_L(y)|^{2} + |Q_2 \Phi_{1,\ldots,k}(y)|^2 )^{\frac{1}{2} \frac{n(n-k-2)}{n-k}}}  &  dy \\
 \leq C' (|Q_1| |Q_2|)^{-\frac{n-k-2}{n-k}} & |\Phi^{(k)}(E)|^{\frac{2}{n-k}} |\varphi^{(d-n+2)}(E)| \end{split} \]
by the usual means of breaking the dyadic sum into one half where $r^k |Q_2|^{-1} \leq |\Phi^{(k)}(E)|$ and another half where the reverse inequality holds.  The term at the boundary of these two halves is the largest term of all (and there is exponential decay away from this maximum in both directions), so the value of the sum is comparable to the value of the largest term.  Now for any $Q \in \GL(n,\R)$, it is possible to find $Q_1, Q_2$ and $L$ such that $|Q(\Phi,1)|^2 = |Q_1 \tilde \Phi_L(y)|^{2} + |Q_2 \Phi^{(k)}(y)|^2$ and $|Q| = |Q_1| |Q_2|$ by an argument analogous to the case when $Q_2$ is a constant (which appeared in the proof of \ref{srlemma}).  Thus, the lemma is complete.
\end{proof}

Finally, the following lemma adapts lemma \ref{codim2klemma} to the form required for theorem \ref{geomthm2}:
\begin{lemma}
Fix $j \geq 0$ and $V = \{v_1,\ldots,v_j\}$.
Suppose $\Phi : U \rightarrow \R^{n}$ and $\varphi : U \rightarrow \R^{d-n+j+1}$ are $C^2$ functions.  Fix an integer $m$ between $0$ and $n-j-1$ (inclusive) and consider the weight $\widetilde{W}^2$ given by \label{codim2jklemma}
\begin{align}
 {\overline{W}}^2_y (\Phi, \varphi) & := \det \left[ \begin{array}{ccc} 0 & 0 & \partial_y \varphi^{(d-n+j+1)}  \\
0 & 0 & \partial_y \Phi^{(m)} \\
(\partial_y \varphi^{(d-n+j+1)})^T & (\partial_y \Phi^{(m)})^T & \overline M \end{array}  \right], \label{weight2jkpt1} \\
{\overline M}_{i l} & := \det \left[ \begin{array}{ccccc} 0 & \cdots & 0 & 0  & \partial_y \varphi^{(d-n+j+1)} \\
v_1 & \cdots & v_j & \frac{\partial^2 \Phi}{\partial y_i \partial y_l} & \partial_y \Phi \end{array} \right]. \label{weight2jkpt2}
\end{align}
  Then
\begin{equation}
\begin{split}
  \sup_{Q \in \GL(n-j,\R)} & |Q|^{\frac{n-j-m-1}{n-j-m+1}} \! \int_E  \frac{|\overline{W}^2_y(\Phi,\varphi)|^{\frac{1}{n-j-m+1}} dy }{1+ |Q T_V \Phi(y)|^{\alpha}}   \\ & 
\leq C' |\Phi^{(m)}(E)|^{\frac{2}{n-j-m+1}}  |\varphi^{(d-n+j+1)}(E)| \end{split}
\label{weight2jk}
\end{equation}
for some topological constant $C'$, where $\alpha = \frac{(n-j+1)(n-j-m-1)}{n-j-m+1}$.
\end{lemma}
\begin{proof}
This is an immediate consequence of lemma \ref{codim2klemma} applied to $\tilde \Phi := T_V \Phi$ and renaming $k$ as $m$.
\end{proof}

\subsection{Codimension 3}

This section is concluded with a computation of the codimension 3 weight.  The construction employed by theorem \ref{gentheorem} requires the designation of a single special direction (indicated by the first coordinate direction) to complete the rescaling argument.  In the passage from codimension 1 to codimension 2, the choice of direction was irrelevant; however, in the codimension 3 case, the choice of a special direction does not entirely cancel out; as a result, the codimension 3 weight will depend on a single, arbitrary (but fixed) linear functional $\ell$:
\begin{lemma}
Fix $U \subset \R^d$ open, $\Phi : U \rightarrow \R^{n}$ and $\varphi : U \rightarrow \R^{d-n+3}$ smooth (suppose also that $n > 2$).  Let $\ell$ be any fixed linear functional on $\R^n$.  Consider the weight $W^3_{y}$ given by
\[ W^3_{y}(\Phi,\varphi) := \det \left[ \begin{array}{cc} 0 & \partial_y \varphi^{(d-n+3)} \\ (\partial_y \varphi^{(d-n+3)})^T & N \end{array} \right], \]
\[
\begin{split} N_{ij} := \det & \left( \left[ \begin{array}{cc} 0 & \partial_y \varphi^{(d-n+2)} \\ 0 & \ang{ \partial^2_{ij} D \Phi, \partial_y \Phi} \\ (\partial_y \varphi^{(d-n+3)})^T & \ang{D \Phi, \partial^2_y \Phi} \end{array} \right] \right.  \\
&  -  \frac{\ang{D \Phi,\partial^2_{ij} \Phi} }{\ell(\Phi)}  \left. \left[ \begin{array}{cc} 0 & \partial_y \varphi^{(d-n+2)} \\ 0 & \partial_y \ell(\Phi) \\ (\partial_y \varphi^{(d-n+3)})^T & \ang{D \Phi, \partial^2_y \Phi} \end{array} \right] \right),
\end{split}
\]
where $W_y^3(\Phi,\varphi) := 0$ at the point $y$ if $\ell(\Phi(y)) = 0$.
\begin{equation}
\begin{split} \sup_{Q \in \GL(n,\R)} |Q|^{\frac{n-3}{n-1}} \int_{E} & \frac{|W_y^3(\Phi,\varphi)|^{\frac{1}{(n-1)(n-2)}}}{|Q \Phi(y)|^{\frac{n(n-3)}{n-1}}} dy \\ & \leq C |\varphi^{(d-n+2)}(E)|^{\frac{n-3}{n-2}} |\varphi^{(d-n+3)}(E)|^{\frac{1}{n-2}} 
\end{split} \label{boxest3}
\end{equation}
for some topological constant $C$.
\label{codim3lemma}
\end{lemma}
\begin{proof}
As in the codimension 2 case, the form of \eqref{boxest3} is verifiable from \eqref{boxest2} and theorem \ref{boxest2} by simple arithmetic.  The main content of the lemma is the computation which explicitly identifies $W^3$.  Regarding the linear functional $\ell$, one may assume without loss of generality that it is the evaluation of the first coordinate of $\Phi$ (since all the expressions are invariant under rotations).
Computing $M_{ij}$ with the pair $(\lambda^{n-1} D_{-} \Phi, \tilde \varphi)$ gives
\begin{align*}
 M_{ij} & = \det \left[ \begin{array}{cccc} 0 & 0 & 0 & \partial_y \varphi^{(d-n+2)} \\
0 & 0 & \Phi_1 & \lambda \partial_y \Phi_1 \\
\partial^2_{ij} (\lambda^{n-1} D_{-} \Phi) & \lambda^{n-1} D_{-} \Phi & \partial_\lambda (\lambda^{n-1} D_{-} \Phi) & \partial_y ( \lambda^{n-1} D_{-} \Phi) \end{array} \right]  \\
& = \pm \lambda^{n(n-1)} \Phi_1  \det \left[ \begin{array}{ccc} 0 & 0 & \partial_y \varphi^{(d-n+2)} \\
\partial^2_{ij} D_{-} \Phi & D_{-} \Phi & \partial_y  D_{-} \Phi \end{array} \right]
\end{align*}
(for some fixed choice of sign) by row reduction.  To see this, observe that the derivatives $\partial_i$ and $\partial_j$, which are linear combinations of the $\frac{\partial}{\partial y_i}$'s and $\frac{\partial}{\partial \lambda}$ may by the usual choice of coordinates, be assumed to annihilate $\tilde \varphi^{(d-n+3)}$.  Thus $\partial^2_{ij} (\lambda^{n-1} D_{-} \Phi) = \lambda^{n-1} \partial^2_{ij} (D_{-} \Phi) + \partial_{i} (\lambda^{n-1}) \partial_j D_{-} \Phi + \partial_{j} (\lambda^{n-1} ) \partial_i D_{-} \Phi + D_{-} \Phi \partial^2_{ij} (\lambda^{n-1})$, but in this particular coordinate system the final three vectors will lie in the span of the final columns of the matrix defining $M_{ij}$.  Likewise, $\partial_\lambda (\lambda^{n-1} D_{-} \Phi) = (n-1) \lambda^{n-2} D_{-} \Phi$, which is a multiple of the $\lambda^{n-1} D_{-} \Phi$ appearing in the column immediately to the left.  Thus the factor $\lambda^{n-1}$ clearly pulls out of the final $n$ rows after the determinant is expanded in the third column.

Since each term of $M_{ij}$ has a natural factor of $\lambda^{n(n-1)} \Phi_1$, the final $d+1$ rows of the outermost matrix corresponding to $W^2_{\lambda,y}( \lambda^{n-1} D_{-} \Phi, \tilde \varphi)$ will have factors of $\lambda^{n(n-1)} \Phi_1$ except in the first $d-n+3$ columns.  Thus, exploiting the block form of this matrix gives that $W^2_{\lambda,y}( \lambda^{n-1} D_{-} \Phi, \tilde \varphi)$ must equal (up to sign) 
\begin{align*}
(\lambda^{n(n-1)} \Phi_1^2)^{n-2}  \det \left[ \begin{array}{cc} 0 & \partial_{\lambda,y} \tilde \varphi^{(d-n+3)} \\ 
(\partial_{\lambda,y} \tilde \varphi^{(d-n+3)})^T & N \end{array} \right], \\
N_{ij} := \frac{1}{\Phi_1} \det \left[ \begin{array}{ccc} 0 & 0 & \partial_y \varphi^{(d-n+2)} \\
\partial^2_{ij} D_{-} \Phi & D_{-} \Phi & \partial_y  D_{-} \Phi \end{array} \right].
\end{align*}
Changing the coordinates of the outermost matrix back to the standard $\frac{\partial}{\partial \lambda}$ and $\frac{\partial}{\partial y_i}$ directions, it is clear that the row and column of $N$ corresponding to $\frac{\partial}{\partial \lambda}$ will vanish (because $D_{-} \Phi$ does not depend on $\lambda$).  Thus there is a row and a column of the outermost matrix, each of which has only one nonzero entry, namely $\Phi_1$ (which arises as $\frac{\partial}{\partial \lambda} (\lambda \Phi_1 - 1)$). Expanding in this row and this column give $W^2_{\lambda,y}( \lambda^{n-1} D_{-} \Phi, \tilde \varphi)$ up to sign
\begin{align*}
(\lambda^{n(n-1)} \Phi_1^2)^{n-2} \Phi_1^2 \det \left[ \begin{array}{cc} 0 & \partial_y \varphi^{(d-n+2)} \\ (\partial_y \varphi^{(d-n+2)})^T & N \end{array} \right], 
\end{align*}
where the indices of $N$ are now understood to refer to $y$-differentiation only.

Regarding the entries $N_{ij}$, they may be expressed via \eqref{detred} as
\[ \frac{1}{\Phi_1} \det \left[ \begin{array}{ccc}
0 & 0 & \partial_y \varphi^{(d-n+1)} \\
0 & \Phi_1 & \partial_y \Phi_1 \\
0 & \ang{ \partial^2_{ij} D_{-} \Phi, \hat \Phi} & \ang{\partial^2_{ij} D_{-} \Phi, \partial_y \hat \Phi} \\
(\frac{\partial}{\partial y_1} \varphi^{(d-n+2)})^T & \ang{\frac{\partial}{\partial y_1}  D_{-} \Phi, \hat \Phi} & \ang{\frac{\partial}{\partial y_1} D_{-} \Phi, \partial_y \hat \Phi} \\
\vdots & \vdots & \vdots \\
(\frac{\partial}{\partial y_d} \varphi^{(d-n+2)})^T & \ang{\frac{\partial}{\partial y_d}  D_{-} \Phi, \hat \Phi} & \ang{\frac{\partial}{\partial y_d} D_{-} \Phi, \partial_y \hat \Phi}
\end{array} \right]. \]
As in the codimension $2$ case, the second row may be added to subsequent rows to conclude that this determinant also equals
\[ \frac{1}{\Phi_1} \det \left[ \begin{array}{ccc}
0 & 0 & \partial_y \varphi^{(d-n+1)} \\
0 & \Phi_1 & \partial_y \Phi_1 \\
0 & \ang{ \partial^2_{ij} D \Phi,  \Phi} & \ang{\partial^2_{ij} D \Phi, \partial_y \Phi} \\
(\frac{\partial}{\partial y_1} \varphi^{(d-n+2)})^T & \ang{\frac{\partial}{\partial y_1}  D \Phi,  \Phi} & \ang{\frac{\partial}{\partial y_1} D \Phi, \partial_y  \Phi} \\
\vdots & \vdots & \vdots \\
(\frac{\partial}{\partial y_d} \varphi^{(d-n+2)})^T & \ang{\frac{\partial}{\partial y_d}  D \Phi,  \Phi} & \ang{\frac{\partial}{\partial y_d} D \Phi, \partial_y  \Phi}
\end{array} \right]. \]
As before, the vector with entries $\ang{\frac{\partial}{\partial y_i} D \Phi,\Phi}$ lies in the span of the gradients of $\varphi_1,\ldots,\varphi_{d-n+1}$, so row reduction admits the simplification
\[ N_{ij} = \frac{1}{\Phi_1} \det \left[ \begin{array}{ccc}
0 & 0 & \partial_y \varphi^{(d-n+1)} \\
0 & \Phi_1 & \partial_y \Phi_1 \\
0 & \ang{ \partial^2_{ij} D \Phi,  \Phi} & \ang{\partial^2_{ij} D \Phi, \partial_y \Phi} \\
(\frac{\partial}{\partial y_1} \varphi^{(d-n+2)})^T & 0 & \ang{\frac{\partial}{\partial y_1} D \Phi, \partial_y  \Phi} \\
\vdots & \vdots & \vdots \\
(\frac{\partial}{\partial y_d} \varphi^{(d-n+2)})^T & 0 & \ang{\frac{\partial}{\partial y_d} D \Phi, \partial_y  \Phi}
\end{array} \right]. \]
Expanding in the column containing $\Phi_1$ and $\ang{ \partial^2_{ij} D \Phi,  \Phi}$ expresses $N_{ij}$ exactly as was done in the statement of the lemma (up to a fixed sign which does not depend on $i,j$--recall also that the statement of the lemma applies when $\Phi$ maps into $\R^n$ but the present calculation concerns $\Phi$ mapping into $\R^{n+1}$).  Thus we have the weight $W'_y$ from theorem \ref{gentheorem} must equal
\[ |W'_y(\Phi,\varphi)| = |\Phi_1|^{(n^2-1)(n-2)-n(n-1)} (\Phi^{-n(n-1)} \Phi_1^2)^{n-2} \Phi_1^2 |W_y^3(\Phi,\varphi)|, \]
which is simply equal to $|W^3_y(\Phi,\varphi)|$.
\end{proof}

\section{More Applications}
\label{finalsec}

Theorems \ref{theorem1} and \ref{theorem2} now follow trivially from theorem \ref{geomthm} and lemmas \ref{codim1lemma} and \ref{codim2lemma}.  For simplicity, let the weight $\ang{D_1 \Phi,\Phi}$ be denoted $W^1_y(\Phi,\varphi)$.  Consider the variable $x$ to be a parameter, and apply lemmas \ref{codim1lemma} and \ref{codim2lemma} with $\Phi := \partial_x \rho$ and $\varphi$ augmented by $\rho$, i.e., apply the lemmas to the pair $\Phi,(\rho,\varphi)$.  The following inequalities must hold uniformly for some topological constant $C$:
\begin{align*}
 \int_E  |W^1(\partial_x & \rho, (\rho,\varphi)) |^{\frac{1}{d_l + 1}} |f(x) g(y)| dx dy \\
  & \leq  C |\rho(E)|_L^{\frac{d_l}{d_l+1}} |\rho(E)|_R^{\frac{1}{d_l+1}} |\varphi^{(d_r-d_l)}(E)|_R^{\frac{1}{d_l+1}} ||f||_{\frac{d_l + 1}{d_l}} ||g||_{\frac{d_l+1}{d_l}}, \\
 \int_E |W^2(\partial_x & \rho, (\rho,\varphi)) |^{\frac{1}{d_l(d_l -1)}} |f(x) g(y)| dx dy \\
  & \leq  C |\rho(E)|_L^{\frac{d_l-2}{d_l-1}} |\rho(E)|_R^{\frac{1}{d_l-1}} |\varphi^{(d_r-d_l+1)}(E)|_R^{\frac{1}{d_l-1}} ||f||_{\frac{d_l(d_l-1)}{d_l^2 - 2d_l + 2}} ||g||_{\frac{d_l-1}{d_l-2}}. 
\end{align*}
These inequalities are, in fact, sharper forms of the conclusions of theorems \ref{theorem1} and \ref{theorem2} (subject to the elementary verification that the weights \eqref{radonrotcurv} and \eqref{radonrotcurv2} agree with the weights given above).

Likewise, theorem \ref{geomthm2} may be applied with lemmas \ref{codim1jlemma} and \ref{codim2jklemma}.  As before, let $\Phi := \partial_x \rho$ and augment $\varphi$ so that it becomes $(\rho,\varphi)$.  For fixed $x$ and $y_1$, let $v_1,\ldots,v_{k+1}$ be given by $\partial_x \varphi_1(x,y_1),\ldots,\partial_x \varphi_{k}(x,y_1)$ and $\partial_x \rho(x,y_1)$ as specified by theorem \ref{geomthm2} and consider the weight
\begin{equation} w(x,y) := \inf_{y_1 \in U_x} |\widetilde{W}^1_y(\partial_x \rho, (\rho,\varphi))| \label{box1kest} \end{equation}
where $\widetilde{W}^1_y$ is as specified by lemma \ref{codim1jlemma} (note in this case that $j=k+1$).  It follows that
\begin{equation}
\begin{split}
 \int_E  |w(x,y)& |^{\frac{1}{d_l -k+ 1}}  |f(x) g(y)| dx dy \leq C |\rho(E)|_L^{\frac{d_l-k}{d_l-k+1}} |\rho(E)|_R^{\frac{1}{d_l-k+1}}\\
   \times  & \left(|\varphi^{(k)}(E)|_L |\varphi^{(d_r-d_l+k)}(E)|_R \right)^{\frac{1}{d_l-k+1}} ||f||_{\frac{d_l -k + 1}{d_l-k}} ||g||_{\frac{d_l-k+1}{d_l-k}}. 
\end{split} \label{cuccagna}
\end{equation}
for some topological constant $C$.
Finally, note that for the weight
\begin{equation} \overline{w}(x,y) := \inf_{y_1 \in U_x} |\overline{W}^2_y(\partial_x \rho, (\rho,\varphi))| \label{box2kest} \end{equation}
with $\overline{W}^2_y$ being given by lemma \ref{codim2jklemma} (and again taking $j = k+1$) yields the estimate
\begin{equation}
\begin{split}
 \int_E  |\overline{w}&(x,y) |^{\alpha}  |f(x) g(y)| dx dy \leq C |\rho(E)|_L^{\frac{\overline{s}}{\overline{s}+1}} |\rho(E)|_R^{\frac{1}{\overline{s}+1}} |\partial_x \rho^{(m)}(E)|_R^{2 \alpha} \\
   \times  & \left(|\varphi^{(k)}(E)|_L^{\frac{\overline{s}}{d_{\mathrm{eff}}}} |\varphi^{(d_r-d_{\mathrm{eff}}+1)}(E)|_R \right)^{\frac{1}{\overline{s}+1}} ||f||_{\frac{d_{\mathrm{eff}}(\overline{s}+1)}{d_{\mathrm{eff}}(\overline{s}+1) - \overline{s}}} ||g||_{\frac{\overline{s}+1}{\overline{s}}}
\end{split}\label{gsfull}
\end{equation}
where
\[ \alpha := \frac{1}{(d_{\mathrm{eff}} - m ) (\overline{s}+ 1)}, \ \  \overline{s} :=  \frac{(d_{\mathrm{eff}} - m - 2)d_{\mathrm{eff}}}{d_{\mathrm{eff}} - m} = d_{\mathrm{eff}}-2 - \frac{2m}{d_{\mathrm{eff}} - m}, \]
 and $d_{\mathrm{eff}} := d_l -k$.

To conclude, write $x = (x',x_{d_l})$ for $x' \in \R^{d_l-1}$ and suppose that the phase function $\rho$ is of the form
\[ \rho(x,y) = - x_{d_l} + s(x',y) \]
and that the functions comprising $\varphi$ do not depend on $x_{d_l}$ at all.  Suppose as well that $\varphi_i(x,y) = x_i + y_i$ for $i=1,\ldots,k$.  Under these circumstances, substantial simplifications of the weights \eqref{box1kest} and \eqref{box2kest} occur.  In this case, if one takes $x'' := (x_{k+1},\ldots,x_{d_l-1})$, $y'' :=(y_{k+1},\ldots,y_{d_r})$, and $\varphi''(x',y) := (\varphi_{k+1}(x',y),\varphi_{k+2}(x',y),\ldots)$, then $|w(x,y)|$ simply equals the Phong-Stein rotational curvature of $\rho$ when $x^{(k)}$ and $y^{(k)}$ are regarded as constants%%  the restricted incidence relation corresponding to $x'', y''$ and $\rho$ being held constant
, i.e., $|w(x,y)| = |W^{1}_{x'',y''}|$.  Thus \eqref{cuccagna} has the following consequence in the spirit of Cuccagna \cite{cuccagna1996}, who proved local $L^p$-Sobolev estimates for FIOs under similar circumstances:
\begin{theorem}
Suppose $\rho$ is real analytic and $d_r = d_l$. Suppose that at some point $(x,y) \in U$, one has $\partial_x \rho(x,y) \neq 0$ %, $\partial_y \rho(x,y) \neq 0$, 
and the rank of the rotational curvature matrix \eqref{radonrotcurv} is at least $r$.  Then the Radon-like operator $R^0$ given by \eqref{radondef} maps $L^{r/(r-1)}_{comp}$ to $L^r_{loc}$ sufficiently near $(x,y)$.
\end{theorem}
Likewise, if $\varphi_{k+1}(x',y) = w(x,y)$, then $\overline{w}(x,y)$ measures the Gaussian curvature of the mapping $y'' \mapsto (\partial_{x''} s)$ when restricted to the submanifold where $s$, $w(x',y)$, and the first $m$ components of $\partial_{x''} s$ are constant.  Again, if the image of the submanifold on which $s$ and $w$ are constant has at least $d_{\mathrm{eff}}-m-2$ nonvanishing principal curvatures, then there will be an appropriate coordinate system for which $\overline{w}(x',y) \neq 0$.  Thus the following generalization of theorem \ref{greenleafseeger} is a consequence of \eqref{gsfull}:
%% Suppose $d_l = d_r$.  Let ${\mathcal M} := \set{(x,y) \in U}{\rho(x,y) = 0}$ for some real analytic $\rho$.  Fix a point $(x,y) \in {\mathcal M}$, and let $X$ and $Y$ be affine $d_{\mathrm{eff}}$-dimensional subspaces of $\R^{d}$ containing $x$ and $y$, respectively.  Let ${\mathscr{C}}$ be the canonical relation associated to ${\mathcal M} \cap X \times Y$.
\begin{theorem}
Suppose $\pi_L : M \rightarrow T^* X$ suppose that at $(x,y)$, $\pi_L$ has a Whitney fold singularity and the image of the fold hypersurface under $\pi_L$ (for $x$ held fixed) has at least $\ell$ nonvanishing principal curvatures.  Then the associated Radon transform $R^0$ maps $L^{q,1}_{comp}$ to $L^{p,\infty}_{loc}$ for $q := \frac{d_{\mathrm{eff}}+1}{d_{\mathrm{eff}}} + \frac{1}{d_{\mathrm{eff}}(\ell + 1)}$ %1 + \frac{\ell + 2}{d_{\mathrm{eff}}(\ell + 1)}$ 
and $p := d_{\mathrm{eff}} q$.
\end{theorem}
\begin{proof}
The proof follows from the same application of Bourgain's trick that was used in theorem \ref{greenleafseeger}.  In particular,
\[ \chi_{\overline{w} \geq \beta} \leq |\epsilon^{-1} w|^{\frac{1}{d_{\mathrm{eff}}+1}} + |\beta^{-1} \overline{w}|^{\alpha} \chi_{w \leq \epsilon}, \]
so by \eqref{cuccagna} and \eqref{gsfull} (and the passage from sublevel operator to Radon-like transform) it follows that
\[ \int_F R^0 \chi_G(x) dx \leq C_\beta \left( \epsilon^{-\frac{1}{d_{\mathrm{eff}} + 1}} |F|^{\frac{d_{\mathrm{eff}}}{d_{\mathrm{eff}}+1}}|G|^{\frac{d_{\mathrm{eff}}}{d_{\mathrm{eff}}+1}}   +   \epsilon^{\frac{1}{\overline{s}+1}} |F|^{1 - \frac{  \overline{s}}{d_{\mathrm{eff}}(\overline{s}+1)}}|G|^{\frac{\overline{s}}{\overline{s}+1}} \right).\]
Optimizing over the choice of $\epsilon$ gives the desired $L^{q,1} \rightarrow L^{p,\infty}$ result (using, of course, $\ell := d_{\mathrm{eff}} - m - 2$).
\end{proof}

\section{Appendix}
\label{appendix}
Suppose $\rho(x_0,y_0) = \rho_0$ and $\nabla_y \rho(x_0,y_0) \neq 0$.  By the implicit function theorem, there exists a $C^1$ mapping $\Psi : D \rightarrow \R^{d_l} \times \R^{d_r}$ such that $D$ is an open ball centered at $(x_0,\rho_0, 0) \in \R^{d_l} \times \R \times \R^{d_r-1}$, $\Psi$ is $1-1$ on $D$ (with a Jacobian determinant uniformly bounded away from $0$ and $\infty$) and $\rho \circ \Psi(x_0, c, t) = c$.  %Let $\overline{D}$ be the open neighborhood of $(x_0,y_0)$ which is the image of the domain of $\Phi$.  
Let $\pi_R : \R^{d_l} \times \R^{d_r} \rightarrow \R^{d_r}$ be given by $\pi_R(x,y) := y$.
%If $\Phi_R$ is the projection of $\Phi$ onto the factor $\R^{d_r}$ and
The notation $\left| \frac{\partial \pi_R \circ \Psi }{\partial t} (x,c,t) \right|$ is meant to denote the square root of the Gram determinant generated by the vectors $\frac{\partial \pi_R \circ \Psi}{\partial t_i}$ for $i=1,\ldots,d_r-1$, then the integral of interest over $\Sigma_x^{R,c} \cap \Psi(D)$ is given exactly by
\begin{align*} \int_{\Sigma^{R,c}_x \cap \Psi(D)} g(y) & \chi_{E} (x,y) \frac{d {\mathcal H}^{d_r-1}(y)}{|\partial_y \rho(x,y)|} \\
& = \int_{D_{x,c}} g( \pi_R ( \Psi(x,c,t))) \chi_{E}(\Psi(x,c,t))  \frac{\left| \frac{\partial \pi_R \circ \Psi}{\partial t} (x,c,t) \right|}{|\partial_y \rho(\Psi(x,c,t))|} dt,
\end{align*}
where $D_{x,c}$ is the subset of $D$ on which $x$ and $c$ have given, fixed values.
Assuming for the moment that $E$ is open and $g$ is positive and continuous, then the integral on the right-hand side will be a lower semi-continuous function of $c$ since the domain is bounded and open (so $\chi_{D \cap \Psi^{-1}(E)}(x,c,t) \rightarrow \chi_{D \cap \Psi^{-1}(E)}(x,c_0,t)$ as $c_0 \rightarrow c$ at every point for which $\chi_{D \cap \Psi^{-1}(E)}(x,c_0,t)=1$) and all the remaining functions in the integrand are positive, continuous functions of $c$. % except for the characteristic function (which nevertheless converges in $t$ pointwise a.e. as $c \rightarrow c_0$ since the boundary of $D$ is a sphere for any fixed $x,c$).

Now let $\varphi$ be a nonnegative partition of unity on $E \setminus \set{(x,y)}{\partial_y \rho(x,y) = 0}$ adapted to the family of balls above provided by the implicit function theorem.  If $f$ is nonnegative, then the monotone convergence theorem guarantees that
\begin{align*}
 \sum_{i} \int f(x) & \int_{\Sigma^{R,c}_x} g(y) \varphi_i (x,y) \frac{d {\mathcal H}^{d_r-1}(y)}{|\partial_y \rho(x,y)|} dx \\
& = \int f(x) \int_{\Sigma^{R,c}_x} g(y) \chi_E(x,y) \chi_{\partial_y \rho(x,y) \neq 0} \frac{d {\mathcal H}^{d_r-1}(y)}{|\partial_y \rho(x,y)|} dx 
\end{align*}
for any value of $c$.  Thus, for any fixed value of $c$, the value of the right-hand side may be approximated arbitrarily well by a finite sum over the partition (meaning within any prescribed $\epsilon$ if the right-hand side is finite or larger than any fixed $N$ if the right-hand side is infinite).  Since each of the terms is a lower semi-continuous function of $c$, the result is that the right-hand side must be a lower semi-continuous function of $c$.  Consequently
\[ \int f(x) R^0 g(x) dx \leq \liminf_{\epsilon \rightarrow 0^{+}} \frac{1}{2 \epsilon} \int_{-\epsilon}^{\epsilon} \int f(x) R^c g(x) dx dc. \]
On the other hand, the change-of-variables formula dictates that
\begin{align*} 
\int_{-\epsilon}^\epsilon  \int f(x) \int g (\pi_R (\Psi (x,c,t))) & \varphi_i (\Psi(x,c,t)) \frac{\left| \frac{\partial \pi \circ \Psi_R}{\partial t} (x,c,t) \right|}{|\partial_y \rho (\Psi(x,c,t))|} dt dx dc \\
& = \int f(x) g(y) \varphi_i(x,y) \chi_{|\rho(x,y)| \leq \epsilon} dx dy
\end{align*}
since $\chi_{|c| \leq \epsilon} = \chi_{|\rho \circ \Psi| \leq \epsilon}$ and
\[ \left| \frac{\partial \Psi}{\partial(x,c,t)} \right| = \frac{\left| \frac{\partial \pi_R \circ \Psi}{\partial t} (x,c,t) \right|}{|\partial_y \rho( \Psi(x,c,t))|} \]
which is a consequence of the following elementary manipulation of determinants.  By placing the determinant in block form, it clearly must be true that $\left| \frac{\partial \Psi}{\partial(x,c,t)} \right| = \left|\frac{\partial \pi_R \circ \Psi}{\partial(c,t)} \right|$.  Now note that $\ang{\partial_y \rho, \frac{\partial \pi_R \circ \Psi}{\partial c}} = 1$ and $\ang{ \partial_y \rho, \frac{\partial \pi_R \circ \Psi}{\partial t_i}} = 0$, meaning that one may apply an orthogonal transformation to the rows of $\frac{\partial \pi_R \circ \Psi}{\partial(c,t)}$ such that the top row vanishes everywhere except for the column corresponding to differentiation in the $c$ direction, where the entry must have magnitude $|\partial_y \rho|^{-1}$.

Finally, note that the passage from $E$ open to $E$ closed is accomplished by outer regularity of the Lebesgue measure.

\bibliography{mybib}

\end{document}